\documentclass[11pt,reqno]{amsart}

\usepackage[T1]{fontenc}
\usepackage{lmodern}
\usepackage{microtype}
\microtypesetup{expansion=false}
\usepackage{amsmath,amssymb,amsthm,mathtools}
\usepackage{booktabs}
\usepackage{enumitem}
\usepackage{xcolor}
\usepackage[colorlinks=true,linkcolor=blue!40!black,citecolor=blue!40!black,
urlcolor=blue!40!black]{hyperref}
\hypersetup{
 pdftitle={A note on the noncommutative Hodge conjecture for graded matrix factorizations},
 pdfauthor={Xun Lin and Shizhuo Zhang}
}

\numberwithin{equation}{section}
\newtheorem{theorem}{Theorem}[section]
\newtheorem{proposition}[theorem]{Proposition}
\newtheorem{lemma}[theorem]{Lemma}
\newtheorem{corollary}[theorem]{Corollary}
\theoremstyle{definition}

\theoremstyle{remark}
\newtheorem{remark}[theorem]{Remark}

\newcommand{\C}{\mathbb C}
\newcommand{\Q}{\mathbb Q}
\newcommand{\Z}{\mathbb Z}
\newcommand{\PP}{\mathbb P}
\newcommand{\A}{\mathbb A}
\newcommand{\Jac}{\operatorname{Jac}}
\newcommand{\Hdg}{\operatorname{Hdg}}
\newcommand{\HH}{\operatorname{HH}}
\newcommand{\HN}{\operatorname{HN}}
\newcommand{\HP}{\operatorname{HP}}
\newcommand{\MF}{\operatorname{MF}}
\newcommand{\MFgr}{\operatorname{MF}_{\mathrm{gr}}}

\newcommand{\Dgrsg}{\operatorname{D}_{\mathrm{sg}}^{\mathrm{gr}}}
\newcommand{\Perf}{\operatorname{Perf}}
\newcommand{\CH}{\operatorname{CH}}
\newcommand{\Conf}{\operatorname{Conf}}
\newcommand{\rk}{\operatorname{rk}}

\newcommand{\pr}{\operatorname{pr}}
\newcommand{\Span}{\operatorname{Span}}
\newcommand{\Hess}{\operatorname{Hess}}
\newcommand{\GL}{\operatorname{GL}}
\newcommand{\id}{\operatorname{id}}

\title[Hodge classes for split singularities]
{A note on the noncommutative Hodge conjecture for graded matrix factorizations}
\author{Xun Lin}
\address{School of Science and Engineering, The Chinese University of Hong Kong,
Shenzhen, China}
\email{lin-x18@tsinghua.org.cn}
\author{Shizhuo Zhang}
\address{School of Mathematics, Sun Yat-sen University, Guangzhou, China}
\email{zhangshzh28@mail.sysu.edu.cn}
\date{}

\begin{document}
\begin{abstract}
Let $m\ge2$ and $d\ge7$.  We consider homogeneous polynomials in
$2m+2$ variables of the form
\[
 f=F_0(u_0,v_0)+\cdots+F_m(u_m,v_m),
\]
where the $F_i$ are independently very general squarefree binary forms of
degree $d$.  We prove the rational noncommutative Hodge conjecture for the
dg category $\MFgr(f)$ of graded matrix factorizations.  Its Hochschild
homology is the direct sum of the scalar-invariant Jacobian sector and
$d-1$ one-dimensional point sectors
\cite[Theorem~2.6.1]{PolishchukVaintrobCohFT}.  Boundary--bulk images of explicit
rank-one factorizations generate a lattice of rank $(d-1)^{m+1}$ in the
identity sector, while grading shifts of the stabilized residue field
generate all point sectors.  A reduced-Burau calculation shows that the
identity-sector lattice exhausts the rational Hodge classes at a very
general parameter.  Consequently,
\[
 \dim_{\Q}\Hdg\bigl(\MFgr(f),\Q\bigr)=(d-1)^{m+1}+d-1,
\]
and the rational topological $K$-rank is
\[
 \frac{(d-1)^{2m+2}+d-1}{d}+d-1.
\]
Finally, additivity of the noncommutative Hodge conjecture for
semi-orthogonal decompositions
\cite[Section~3.1, Theorem~3.13 and Remark~3.14]{Lin}, applied
to the decompositions in \cite[Theorem~3.11]{Orlov} for
Fano, Calabi--Yau, and general-type hypersurfaces, proves the rational
Hodge conjecture for the associated smooth projective hypersurface.
\end{abstract}
\maketitle
\tableofcontents

\section{Introduction}

\subsection{The problem}

Let $Y$ be a smooth projective complex variety.  Its Hodge decomposition
gives
\[
 H^{2p}(Y,\C)=\bigoplus_{a+b=2p}H^{a,b}(Y),
\]
and the space of rational Hodge classes of codimension $p$ is
\begin{equation}\label{eq:hodge-classical}
 \Hdg^p(Y,\Q):=H^{2p}(Y,\Q)\cap H^{p,p}(Y).
\end{equation}
The cohomology class of every codimension-$p$ algebraic cycle lies in this
intersection.  The rational Hodge conjecture asserts that the converse
holds: the cycle-class map
\begin{equation}\label{eq:cycle-map}
 \operatorname{cl}_{\Q}:\CH^p(Y)_{\Q}\longrightarrow\Hdg^p(Y,\Q)
\end{equation}
is surjective for every $p$; see \cite[Chapter 7]{Voisin}.  Thus the
conjecture asks whether a topologically rational class that has Hodge type
$(p,p)$ must be represented by a rational linear combination of algebraic
subvarieties.

Now let $S=\C[x_0,\ldots,x_{2m+1}]$, let $f\in S$ be homogeneous with an
isolated critical point at the origin, and set
\[
 X_f=\{f=0\}\subset\PP^{2m+1}.
\]
Then $X_f$ is smooth of dimension $2m$.  By weak Lefschetz and Poincar\'e
duality, its cohomology outside the middle degree is inherited from
projective space and its nonprimitive middle class is the power $h^m$ of
the hyperplane class.  Hence the only nonautomatic part of the rational
Hodge conjecture for $X_f$ is the algebraicity of
\[
 H_{\mathrm{prim}}^{2m}(X_f,\Q)\cap H^{m,m}(X_f).
\]

The same homogeneous polynomial defines the category of graded matrix
factorizations
\begin{equation}\label{eq:graded-category-intro}
 \mathcal G_f:=\MFgr(f)\simeq\Dgrsg(S/(f)).
\end{equation}
By \cite[Theorem~3.10]{Orlov}, the second category in
\eqref{eq:graded-category-intro} is the graded singularity category of
$S/(f)$.  If $f$ has an isolated critical point, then $\mathcal G_f$ is
smooth and proper; this also follows from the semi-orthogonal descriptions
in \cite[Theorem~3.11]{Orlov}.  It therefore carries algebraic and
topological $K$-theory, Hochschild and cyclic homology, and the
noncommutative Hodge structure used below.

By \cite[Theorem~3.11]{Orlov}, $\mathcal G_f$ and $\Perf(X_f)$ occur in a
semi-orthogonal decomposition whose remaining components are exceptional.
Since the noncommutative Hodge conjecture is
additive for semi-orthogonal decompositions
\cite[Section~3.1, Theorem~3.13 and Remark~3.14]{Lin}, proving it for $\mathcal G_f$ implies
the corresponding statement for $\Perf(X_f)$ when $d<2m+2$, $d=2m+2$,
or $d>2m+2$.

There are two complementary formulations of the noncommutative Hodge
conjecture.  For an admissible subcategory of the derived category of a
smooth projective variety, the canonical weight-zero Hodge structure on
topological $K$-theory is constructed in
\cite[Section~5.1, Proposition~5.4]{Perry}.
Its integral and rational Hodge classes are defined in
\cite[Section~5.2, Definition~5.9]{Perry}, and the rational algebraicity
conjecture is \cite[Section~5.2, Conjecture~5.11]{Perry}.  For a general
small dg category, the rational Hodge-class space is defined in
\cite[Section~3.1, Definition~3.1]{Lin}, and the corresponding
noncommutative Hodge conjecture is
\cite[Section~3.1, Conjecture~3.2]{Lin}.  The smooth proper reformulation is
given in \cite[Section~3.1, Definition~3.7]{Lin}, and its equivalence with
the general definition is proved in
\cite[Section~3.1, Remark~3.8]{Lin}.  These formulations combine negative
cyclic homology, periodic cyclic homology, Hochschild homology, and the
topological Chern character constructed in \cite[Sections~3--4]{Blanc}.
For an admissible subcategory of a smooth projective derived category, the
equivalence between these two rational formulations is established in
\cite[Section~3.1, Theorem~3.5]{Lin}.

We use the rational formulation
\cite[Section~3.1, Definition~3.1 and Conjecture~3.2]{Lin} directly for the
smooth proper dg category of graded matrix factorizations.
Negative cyclic homology records the Hodge
filtration, topological $K$-theory records the rational lattice, and the map
to $\HH_0$ extracts the weight-zero part.  The noncommutative Hodge problem
asks whether every vector satisfying these three conditions is the
Hochschild Chern character of an algebraic $K$-class.

The graded category has two different types of summands.  Its identity
sector is the scalar-invariant Jacobian state space, whose rational
type-$(0,0)$ vectors are controlled by braid monodromy of the binary
factors.  The other $d-1$ summands are one-dimensional point sectors.  They
are constant of type $(0,0)$ and are generated algebraically by grading
shifts of the stabilized residue field.  The two parts will be treated
separately throughout the proof.

\subsection{Statement of the results}

For a two-dimensional complex vector space $V$, let
\[
 U_d=\{F\in\operatorname{Sym}^d(V^*):F\text{ has }d
 \text{ distinct zeros on }\PP(V)\},
 \qquad B=U_d^{m+1}.
\]
A point $b=(F_0,\ldots,F_m)\in B$ determines
\begin{equation}\label{eq:split-polynomial}
 f_b=\sum_{s=0}^{m}F_s(u_s,v_s),\qquad
 R_b=\C[u_0,v_0,\ldots,u_m,v_m]/(f_b).
\end{equation}
The variables occurring in distinct summands are disjoint.  A statement for a very general
$b$ means that it holds outside a countable union of proper closed algebraic
subsets of $B$.

\begin{theorem}\label{thm:main-graded}
Let $m\ge2$, $d\ge7$, and let $b\in B$ be very general.  The rational
noncommutative Hodge conjecture holds for the dg category of graded matrix
factorizations
\[
 \mathcal G_b:=\MFgr(f_b)\simeq\Dgrsg(R_b).
\]
More precisely, there is a direct-sum decomposition
\[
 \begin{gathered}
 \Hdg(\mathcal G_b,\Q)=A_{b,\Q}\oplus T_{b,\Q},\\
 \dim_{\Q}A_{b,\Q}=(d-1)^{m+1},
 \qquad \dim_{\Q}T_{b,\Q}=d-1.
 \end{gathered}
\]
such that both summands are generated by Hochschild Chern characters of
graded matrix factorizations.  Consequently,
\begin{align}
 \operatorname{ch}\bigl(K_0(\mathcal G_b)_{\Q}\bigr)
   &=\Hdg(\mathcal G_b,\Q),\label{eq:main-ch}\\
 \dim_{\Q}\Hdg(\mathcal G_b,\Q)
   &=(d-1)^{m+1}+d-1,\label{eq:main-hodge-rank}\\
 \rk_{\Q}K_0^{\mathrm{top}}(\mathcal G_b)_{\Q}
   &=\frac{(d-1)^{2m+2}+d-1}{d}+d-1,
 \qquad K_1^{\mathrm{top}}(\mathcal G_b)_{\Q}=0.
 \label{eq:main-top-rank}
\end{align}
\end{theorem}

Here $\Hdg(\mathcal G_b,\Q)\subset\HH_0(\mathcal G_b)$ is the rational
Hodge-class space defined in \eqref{eq:lin-hodge}.  The unadorned Chern
character means $\operatorname{ch}_{\HH}$.  The summand $A_{b,\Q}$ lies in
the identity sector and is generated by tensor products of rank-one
factorizations.  The summand $T_{b,\Q}$ is the rational form of the $d-1$
point sectors and is generated by grading shifts of the stabilized residue
field.

The projective consequence is stated separately because its proof uses a
semi-orthogonal decomposition and the additivity theorem
\cite[Section~3.1, Theorem~3.13 and Remark~3.14]{Lin}.

\begin{corollary}\label{cor:projective}
Let
\[
 X_b=\{f_b=0\}\subset\PP^{2m+1}
\]
Then $X_b$ is smooth of dimension $2m$ and satisfies the rational Hodge
conjecture.  Its primitive middle Hodge classes have rank $(d-1)^{m+1}$,
and its full middle rational Hodge group has
rank $(d-1)^{m+1}+1$.
\end{corollary}

\subsection{Structure of the proof}

The identity sector and the point sectors are treated separately.  The
semi-orthogonal additivity theorem is used only after the graded categorical
statement has been proved.
\begin{enumerate}[label=\textup{(\roman*)},leftmargin=2.7em]
\item The equivariant Hochschild formula
\cite[Theorem~2.6.1 and Example~2.6.3]{PolishchukVaintrobCohFT}
decomposes $\HH_*(\mathcal G_b)$ into the identity sector
$\Omega_{f_b}^{T=1}$ and $d-1$ one-dimensional point sectors.
\item For each binary factor, $d$ rank-one factorizations are written down
and their boundary--bulk Chern characters are calculated by
\cite[Section~3]{PolishchukVaintrob}.
\item Thom--Sebastiani tensor products produce $(d-1)^{m+1}$ independent
algebraic classes in the identity sector.  Separately, the $d$ grading shifts
of the stabilized residue field have a Fourier matrix of twisted Chern
characters and span all $d-1$ point sectors; see
\cite[Theorem~2.6.1(ii)]{PolishchukVaintrobCohFT} and
\cite[Proposition~4.3.4]{PolishchukVaintrob}.
\item The roots-of-unity projector computes the dimension of the identity
sector.  Since every Hochschild degree is even, the topological Chern
character \cite[Sections~3--4]{Blanc} gives the topological ranks in
\eqref{eq:main-top-rank} after Hodge--de Rham degeneration.
\item The twisted de Rham filtration, the topological $K$-local system and
the singularity pairing give the identity sectors a polarizable rational
Hodge variation.  The algebraicity of Hodge loci \cite{CDK} turns a
rational Hodge class at a very general parameter into a type-$(0,0)$
sub-local-system, hence into a finite-monodromy subspace.  The reduced-Burau
calculation then shows that the classes constructed in step~(iii) exhaust
the rational Hodge classes in the identity sector.  The point sectors are
already generated algebraically, which proves the conjecture for
$\mathcal G_b$.
\item The decompositions \cite[Theorem~3.11]{Orlov} and additivity of
the noncommutative Hodge conjecture
\cite[Section~3.1, Theorem~3.13 and Remark~3.14]{Lin} transfer the result between
$\mathcal G_b$ and $\Perf(X_b)$.  The comparison with the classical
rational Hodge conjecture in \cite[Section~3.1, Remark~3.9]{Lin} proves the corollary.
\end{enumerate}

\subsection{Related work}

For smooth hypersurfaces in projective space, weak Lefschetz reduces the
rational Hodge conjecture to the primitive middle cohomology in even
dimension \cite[Chapter 7]{Voisin}.  Classical positive results include
the cubic-fourfold case \cite{Zucker} and extensive work on diagonal
Fermat hypersurfaces \cite{Shioda,Ran,Aoki}.  In the Fermat setting, the
large diagonal symmetry decomposes primitive cohomology into character
spaces, and special algebraic cycles are constructed inside those spaces.
On the categorical side,
graded matrix factorizations are related to projective hypersurfaces by the
semi-orthogonal decompositions in \cite[Theorem~3.11]{Orlov}, and the
noncommutative Hodge conjecture is additive for such decompositions by
\cite[Section~3.1, Theorem~3.13 and Remark~3.14]{Lin}.  These results reduce the projective
consequence to a direct proof for the graded matrix-factorization component,
but they do not construct the required classes in that component.

The family considered in this paper is a very general family of sums of
independently varying squarefree binary forms.  The proof uses explicit
matrix factorizations and braid monodromy, rather than the diagonal character
decomposition of a Fermat equation.  To the best of our knowledge,
Theorem~\ref{thm:main-graded} and Corollary~\ref{cor:projective} are new:
they establish the rational
noncommutative Hodge conjecture, and consequently the rational Hodge
conjecture for the associated projective hypersurfaces, in every even
dimension at least four, together with the exact ranks stated above.
Section~9 proves that the general members of this family are not linearly
equivalent to Fermat hypersurfaces; thus the examples are genuinely outside
the classical Fermat cases.  Section~10 records their exact categorical and
Hodge-theoretic invariants.

\subsection*{Acknowledgements}

S.Z. acknowledges support from the start-up fund of Sun Yat-sen University
(grant no.~34000-12256019).

\section{Noncommutative Hodge classes}

\subsection{The graded matrix-factorization category}

Let $S=\C[x_1,\ldots,x_N]$ and let $f\in S$.  A matrix factorization of
$f$ is a finite-rank $\Z/2$-graded free $S$-module $E=E^0\oplus E^1$ with
an odd $S$-linear endomorphism $D$ satisfying
$D^2=f\operatorname{id}_E$.  In block form,
\[
 D=\begin{pmatrix}0&B\\A&0\end{pmatrix},
 \qquad AB=BA=fI.
\]
Assume that $f$ is homogeneous of degree $d$ for the standard grading.  The
scalar action of $\mathbb G_m$ on $\A^N$ satisfies
\[
 f(\lambda x)=\lambda^d f(x),
\]
so $f$ is semi-invariant for the character
$\chi_d(\lambda)=\lambda^d$.  The category considered throughout the paper
is
\[
 \mathcal G_f:=\MF_{\mathbb G_m}(\A^N,\chi_d,f)
 =\MFgr(f)\simeq\Dgrsg(S/(f)).
\]
Its objects are $\mathbb G_m$-equivariant factorizations whose differential
squares to the $\chi_d$-semi-invariant function $f$.  The equivalence with
the graded singularity category is established in
\cite[Theorem~3.10]{Orlov}.

The kernel of $\chi_d$ is $\mu_d$.  Accordingly, the Hochschild homology of
$\mathcal G_f$ has one sector for every $\gamma\in\mu_d$.  The identity
sector is obtained from the local Jacobian state space, whereas every
nonidentity scalar fixes only the origin and contributes a one-dimensional
point sector.  The decomposition will be made explicit in
Section~\ref{subsec:graded-sectors}.

\subsection{Two formulations of noncommutative Hodge classes}

Let $\mathcal C$ be a small $\C$-linear dg category.  We write
$\HN_0(\mathcal C)$, $\HP_0(\mathcal C)$ and $\HH_0(\mathcal C)$ for its
negative cyclic, periodic cyclic and Hochschild homology.  The maps needed
in the rational definition are
\begin{equation}\label{eq:cyclic-maps}
 K_0(\mathcal C)\xrightarrow{\operatorname{ch}_{\HN}}\HN_0(\mathcal C)
 \xrightarrow{j}\HP_0(\mathcal C),
 \qquad
 \pi:\HN_0(\mathcal C)\longrightarrow\HH_0(\mathcal C),
\end{equation}
together with the topological Chern character
\cite[Sections 3--4]{Blanc}
\begin{equation}\label{eq:top-ch}
 \operatorname{ch}^{\mathrm{top}}:
 K_0^{\mathrm{top}}(\mathcal C)_{\Q}\longrightarrow\HP_0(\mathcal C).
\end{equation}
Here $j$ inverts the periodicity variable $u$, whereas $\pi$ sets $u=0$.
The first operation passes from the Hodge filtration to periodic cyclic
homology; the second remembers its degree-zero Hochschild term.

Following \cite[Section~3.1, Definition~3.1]{Lin}, define the space of rational Hodge
classes by
\begin{equation}\label{eq:lin-hodge}
 \Hdg(\mathcal C,\Q):=
 \pi\!\left(j^{-1}\!\left(
 \operatorname{ch}^{\mathrm{top}}
 \bigl(K_0^{\mathrm{top}}(\mathcal C)_{\Q}\bigr)\right)\right)
 \subseteq\HH_0(\mathcal C).
\end{equation}
The inverse image is essential: for a general small dg category, the map
$j$ need not be injective.  The rational noncommutative Hodge conjecture
\cite[Section~3.1, Conjecture~3.2]{Lin} is the surjectivity statement
\begin{equation}\label{eq:lin-conjecture}
 \operatorname{ch}_{\HH}:K_0(\mathcal C)_{\Q}\longrightarrow\HH_0(\mathcal C),
 \qquad
 \operatorname{im}(\operatorname{ch}_{\HH})=\Hdg(\mathcal C,\Q).
\end{equation}
Thus the symbol $\Hdg(\mathcal C,\Q)$ used throughout this paper always
means the subspace of $\HH_0(\mathcal C)$ defined in
\eqref{eq:lin-hodge}, and the unadorned Chern character in our theorem
statements means $\operatorname{ch}_{\HH}$.

For comparison, suppose that $\mathcal C$ is smooth and proper and that the
noncommutative Hodge--de Rham spectral sequence degenerates.  Choose a
splitting
\[
 \mathcal H:\HP_0(\mathcal C)\xrightarrow{\sim}
 \bigoplus_{r\in\Z}\HH_{2r}(\mathcal C)
\]
and let $\operatorname{Pr}_0$ denote projection to $\HH_0(\mathcal C)$.
For a smooth proper dg category, the formula becomes
\begin{equation}\label{eq:smooth-proper-hodge}
 \Hdg(\mathcal C,\Q)=
 \operatorname{Pr}_0\mathcal H\!\left(
 \operatorname{ch}^{\mathrm{top}}
 \bigl(K_0^{\mathrm{top}}(\mathcal C)_{\Q}\bigr)
 \cap j\bigl(\HN_0(\mathcal C)\bigr)\right).
\end{equation}
Although a splitting appears in this display, the resulting conjecture is
independent of that choice and is equivalent to \eqref{eq:lin-conjecture};
see \cite[Section~3.1, Remark~3.8]{Lin}.  After the usual Tate normalization,
$j(\HN_0)$ is the $F^0$ part of periodic cyclic homology.  Consequently,
the intersection in \eqref{eq:smooth-proper-hodge} is the periodic
representative of a rational type-$(0,0)$ class, while
$\operatorname{Pr}_0\mathcal H$ places that class in the target $\HH_0$ of
the algebraic Chern character.

We make the type condition explicit.  Let $V_{\Q}$ be a pure rational
Hodge structure of weight zero.  Thus
\begin{equation}\label{eq:weight-zero-decomposition}
 V_{\C}=V_{\Q}\otimes_{\Q}\C
 =\bigoplus_{p\in\Z}V^{p,-p},
 \qquad
 \overline{V^{p,-p}}=V^{-p,p}.
\end{equation}
Its decreasing Hodge filtration is
\begin{equation}\label{eq:weight-zero-filtration}
 F^aV_{\C}=\bigoplus_{p\ge a}V^{p,-p},
 \qquad
 \overline{F^aV_{\C}}
 =\bigoplus_{p\le-a}V^{p,-p}.
\end{equation}
In particular,
\begin{equation}\label{eq:type-zero-zero}
 V^{0,0}
 =F^0V_{\C}\cap\overline{F^0V_{\C}},
 \qquad
 \Hdg(V_{\Q})=V_{\Q}\cap V^{0,0}.
\end{equation}
For $v\in V_{\Q}$, writing
$v_{\C}=\sum_pv^{p,-p}$, one has
\begin{align}
 v\text{ has type }(0,0)
 &\Longleftrightarrow
 v_{\C}\in F^0V_{\C}\cap\overline{F^0V_{\C}}
 \label{eq:type-zero-equivalences}\\
 &\Longleftrightarrow
 v^{p,-p}=0\qquad(p\ne0).\notag
\end{align}

For $V_{\Q}=K_0^{\mathrm{top}}(\mathcal C)_{\Q}$, put
\[
 \kappa=\operatorname{ch}^{\mathrm{top}}(\xi)
 \in\HP_0(\mathcal C),
 \qquad
 \xi\in K_0^{\mathrm{top}}(\mathcal C)_{\Q}.
\]
The Hodge--de Rham comparison gives
\begin{equation}\label{eq:F0-negative-cyclic}
 F^0\HP_0(\mathcal C)=j\bigl(\HN_0(\mathcal C)\bigr).
\end{equation}
Since $\kappa$ is rational,
\[
 \kappa\in F^0
 \quad\Longrightarrow\quad
 \kappa=\overline\kappa\in\overline{F^0}.
\]
Consequently,
\begin{equation}\label{eq:topological-type-zero}
 \begin{aligned}
 \xi\text{ has type }(0,0)
 &\Longleftrightarrow
 \kappa\in
 \operatorname{ch}^{\mathrm{top}}
 \bigl(K_0^{\mathrm{top}}(\mathcal C)_{\Q}\bigr)
 \cap F^0\HP_0(\mathcal C)\\
 &\Longleftrightarrow
 \kappa\in
 \operatorname{ch}^{\mathrm{top}}
 \bigl(K_0^{\mathrm{top}}(\mathcal C)_{\Q}\bigr)
 \cap j\bigl(\HN_0(\mathcal C)\bigr).
 \end{aligned}
\end{equation}
If $\kappa=j(\widetilde\kappa)$ with
$\widetilde\kappa\in\HN_0(\mathcal C)$, then
\[
 \pi(\widetilde\kappa)\in\HH_0(\mathcal C)
\]
is the degree-zero Hochschild component used in
\eqref{eq:lin-hodge}.

The formulation for admissible subcategories starts one step earlier.  By
\cite[Section~5.1, Proposition~5.4]{Perry}, if
$\mathcal C\subset\Perf(X)$ is admissible and $X$ is smooth and projective,
then $K_0^{\mathrm{top}}(\mathcal C)$ carries a natural weight-zero integral
Hodge structure.  The subgroup
\begin{equation}\label{eq:perry-hodge}
 \Hdg_{\mathrm P}(\mathcal C,\Z)=
 \{v\in K_0^{\mathrm{top}}(\mathcal C):
 v_{\C}\text{ has Hodge type }(0,0)\}
\end{equation}
is defined in \cite[Section~5.2, Definition~5.9]{Perry}.  After
rationalization, the associated surjectivity conjecture
$K_0(\mathcal C)_{\Q}\to\Hdg_{\mathrm P}(\mathcal C,\Q)$ is
\cite[Section~5.2, Conjecture~5.11]{Perry}.
Equivalence with \eqref{eq:lin-conjecture} for every such admissible
category is established in \cite[Section~3.1, Theorem~3.5]{Lin}.  We use the general definition
\eqref{eq:lin-hodge} for $\mathcal G_f=\MFgr(f)$; equations
\eqref{eq:F0-negative-cyclic}--\eqref{eq:topological-type-zero} give its
weight-zero interpretation.

\subsection{The twisted de Rham model}

Assume that $f$ is homogeneous of degree $d$ and has an isolated critical
point.  The equivariant Hochschild--Kostant--Rosenberg comparison identifies
the identity sector of the graded category with the invariant twisted de
Rham complexes
\begin{align}
 \HN^{\Z/2}(\mathcal G_f)_{\mathrm{id}}
 &\simeq\bigl(\Omega_S^\bullet[[u]],u\,d+df\wedge-\bigr)^{\mu_d},
 \label{eq:HN-deRham}\\
 \HP^{\Z/2}(\mathcal G_f)_{\mathrm{id}}
 &\simeq\bigl(\Omega_S^\bullet((u)),u\,d+df\wedge-\bigr)^{\mu_d}.
 \label{eq:HP-deRham}
\end{align}
The twisted de Rham comparison, including its Gauss--Manin connection, is
established in \cite{Efimov}; passage to the identity equivariant sector
follows from \cite[Theorem~2.6.1]{PolishchukVaintrobCohFT}.  Before taking
invariants, the cohomology is concentrated in the parity of $N$, and its
associated graded is computed by the Jacobian state space
\cite[Theorem~6.6]{Dyckerhoff}:
\begin{equation}\label{eq:jac-state}
 \Omega_f=
 \frac{\Omega_S^N}{df\wedge\Omega_S^{N-1}}
 \simeq\Jac(f)\,dx_1\wedge\cdots\wedge dx_N,
 \qquad
 \Jac(f)=S/(\partial_1f,\ldots,\partial_Nf).
\end{equation}
Let $\zeta=e^{2\pi i/d}$.  Scalar multiplication by $\zeta$ induces
\[
 T(g\,dx_1\wedge\cdots\wedge dx_N)
 =\zeta^{\deg g+N}g\,dx_1\wedge\cdots\wedge dx_N.
\]
The identity part of the graded theory is therefore
\begin{equation}\label{eq:scalar-invariants}
 \Omega_f^{T=1}=\bigoplus_{\substack{q\in\Z\\q\equiv0\;({\rm mod}\ d)}}
 (\Omega_f)_q.
\end{equation}

\subsection{The sector decomposition of the graded category}
\label{subsec:graded-sectors}

For $\gamma_s=\zeta^s\in\mu_d$, put
\[
 (\A^N)^{\gamma_s}=\{x:\gamma_sx=x\},
 \qquad f_s=f|_{(\A^N)^{\gamma_s}}.
\]
By \cite[Theorem~2.6.1 and Example~2.6.3]{PolishchukVaintrobCohFT},
\begin{equation}\label{eq:equivariant-sector-formula}
 \HH_*(\mathcal G_f)\simeq
 \bigoplus_{s=0}^{d-1}H(f_s)^{\mu_d},
 \qquad
 H(f_s)=\Jac(f_s)\,dx_{(\A^N)^{\gamma_s}}.
\end{equation}
For $s=0$, the fixed space is $\A^N$ and the corresponding summand is
$\Omega_f^{T=1}$.  If $s\ne0$, then scalar multiplication by
$\zeta^s\ne1$ fixes only the origin, so
\[
 (\A^N)^{\gamma_s}=\{0\},
 \qquad f_s=0,
 \qquad H(f_s)=\C\mathbf1_s.
\]
Each $\mathbf1_s$ has Hochschild degree zero.  Hence
\begin{equation}\label{eq:graded-hh-decomposition}
 \HH_*(\mathcal G_f)\simeq\Omega_f^{T=1}\oplus T_{f,\C},
 \qquad
 T_{f,\C}:=\bigoplus_{s=1}^{d-1}\C\mathbf1_s.
\end{equation}
The first summand retains its integral Hochschild grading.  If
$\omega\in(\Omega_f)_{di}$, then
\begin{equation}\label{eq:graded-hh-degree}
 \deg_{\HH}(\omega)=N-2i.
\end{equation}
In particular, for $N=2m+2$,
\begin{equation}\label{eq:graded-hh-zero}
 \HH_0(\mathcal G_f)=(\Omega_f)_{d(m+1)}\oplus T_{f,\C}.
\end{equation}
The second summand is constant in families and has pure type $(0,0)$.  The
first contains the nontrivial variation treated by braid monodromy in
Sections~6 and~7.

\section{Binary factors and algebraic Chern characters}

Let
\begin{equation}\label{eq:binary-factor}
 F(u,v)=\prod_{r=1}^{d}\ell_r(u,v)
\end{equation}
be squarefree and put $h_r=F/\ell_r$.  The pair
\[
 L_r=(\ell_r,h_r),\qquad
 D_r=\begin{pmatrix}0&h_r\\ \ell_r&0\end{pmatrix}
\]
is a rank-one matrix factorization because $D_r^2=FI$.
Since $\deg\ell_r=1$ and $\deg h_r=d-1$, it has the canonical graded lift
\[
 S(-1)\xrightarrow{\ \ell_r\ }S
 \xrightarrow{\ h_r\ }S(d-1),
\]
and hence defines an object of $\MFgr(F)$.

\begin{lemma}\label{lem:binary-chern}
The Chern characters of $L_1,\ldots,L_d$ span a $(d-1)$-dimensional
subspace $A_F\subset\Omega_F$.  Their only linear relation is
\[
 \sum_{r=1}^{d}\operatorname{ch}(L_r)=0.
\]
On the cover on which the roots are ordered, the resulting integral lattice
is the augmentation lattice
\[
 A_F\simeq
 \left\{(c_1,\ldots,c_d)\in\Z^d:\sum_r c_r=0\right\}.
\]
\end{lemma}

\begin{proof}
The boundary--bulk formula in
\cite[Section~3]{PolishchukVaintrob} gives, up to a common nonzero scalar,
\begin{equation}\label{eq:binary-boundary-bulk}
 \operatorname{ch}(L_r)=q_r(u,v)\,du\wedge dv,\qquad
 q_r=\det
 \begin{pmatrix}
 \partial_u\ell_r&\partial_u h_r\\
 \partial_v\ell_r&\partial_v h_r
 \end{pmatrix}.
\end{equation}
Choose the affine chart $v=1$, write $F(x)=\prod_j(x-a_j)$, and take
$\ell_r=u-a_rv$.  Euler's identity for the homogeneous polynomial $h_r$ of
degree $d-1$ yields
\begin{equation}\label{eq:evaluation-qr}
 q_r(a_j,1)=
 \begin{cases}
 (d-1)F'(a_j),&j=r,\\
 -F'(a_j),&j\ne r.
 \end{cases}
\end{equation}
Indeed, at $a_r$ the determinant in
\eqref{eq:binary-boundary-bulk} equals
$a_r\partial_uh_r+\partial_vh_r=(d-1)h_r=(d-1)F'(a_r)$.  If
$j\ne r$, then $h_r(a_j,1)=0$ and Euler's identity gives
$a_j\partial_uh_r+\partial_vh_r=0$, whence the second line of
\eqref{eq:evaluation-qr}.

After division of the $j$th evaluation coordinate by $F'(a_j)$, the
polynomial $q_r$ maps to
\[
 d\,e_r-(1,\ldots,1)\in\C^d.
\]
These $d$ vectors span the augmentation hyperplane and have only the
relation that their sum is zero.  Evaluation at the $d$ distinct points is
injective on binary forms of degree at most $d-2$.  Since each $q_r$ has
degree $d-2$, the same statement holds before evaluation.  Finally, the
Jacobian ideal of $F$ begins in degree $d-1$, so no additional relation is
introduced in $\Omega_F$.  The integral assertion follows from the
root-labelled construction.
\end{proof}

\begin{remark}\label{rem:finite-AF}
The braid group permutes the root labels of $A_F$; hence the monodromy of
$A_F$ is finite.  This elementary finite-monodromy summand will be shown to
be the entire rational finite-monodromy part of the binary state space.
\end{remark}

\section{Thom--Sebastiani and explicit objects}

For $f_b$ as in \eqref{eq:split-polynomial}, disjointness of the variables
gives
\begin{equation}\label{eq:TS-jac}
 \Jac(f_b)\simeq\bigotimes_{s=0}^{m}\Jac(F_s),
 \qquad
 \Omega_{f_b}\simeq\bigotimes_{s=0}^{m}\Omega_{F_s}.
\end{equation}
The categorical Thom--Sebastiani equivalence is proved in \cite{Preygel}.
The corresponding tensor rules for twisted de Rham complexes,
Gauss--Manin connections, topological realizations, and cyclic Chern
characters follow functorially by applying the relevant monoidal invariants.

For a multi-index $I=(i_0,\ldots,i_m)$, form the exterior tensor product
\begin{equation}\label{eq:tensor-factorization}
 E_I=L_{0i_0}\boxtimes\cdots\boxtimes L_{mi_m}.
\end{equation}
Its signed tensor differential has square $f_b$.  The factors are
homogeneous, so \eqref{eq:tensor-factorization} defines both a two-periodic
factorization and an object of $\mathcal G_b=\MFgr(f_b)$.  Multiplicativity
of the boundary--bulk map \cite[Section~3]{PolishchukVaintrob} gives
\begin{equation}\label{eq:tensor-chern}
 \operatorname{ch}(E_I)=
 \prod_{s=0}^{m}q_{s i_s}\,
 du_0\wedge dv_0\wedge\cdots\wedge du_m\wedge dv_m.
\end{equation}

\begin{proposition}\label{prop:algebraic-lattice}
The Chern characters in \eqref{eq:tensor-chern} span the lattice
\[
 A_b=A_{F_0}\otimes\cdots\otimes A_{F_m}
 \subset(\Omega_{f_b})_{d(m+1)}\subset\HH_0(\mathcal G_b),
 \qquad
 \rk A_b=(d-1)^{m+1}.
\]
In particular,
\[
 A_{b,\Q}\subset
 \operatorname{ch}\bigl(K_0(\mathcal G_b)_{\Q}\bigr)
 \subset\Hdg(\mathcal G_b,\Q).
\]
\end{proposition}

\begin{proof}
The rank follows from Lemma~\ref{lem:binary-chern} and the tensor product.
Each $A_{F_s}$ lies in the scalar-invariant binary sector.  Moreover,
\[
 \deg\!\left(\prod_{\nu=0}^{m}q_{\nu i_\nu}\right)
 =(m+1)(d-2),
 \qquad
 \deg(du_0\wedge dv_0\wedge\cdots\wedge du_m\wedge dv_m)=2m+2.
\]
Thus every class in \eqref{eq:tensor-chern} has total weight
\[
 (m+1)(d-2)+(2m+2)=d(m+1).
\]
Equation~\eqref{eq:graded-hh-zero} therefore places $A_b$ in the
degree-zero identity sector of $\HH_0(\mathcal G_b)$.  Compatibility of the
algebraic, negative-cyclic, and topological Chern characters places every
boundary--bulk class in the Hodge subspace \eqref{eq:lin-hodge}; see
\cite[Section~3.1, Definition~3.1]{Lin}.  This proves the stated inclusions.
\end{proof}

\subsection{Algebraic generators of the point sectors}

Write $N=2m+2$ and enumerate the variables of $f_b$ as
$x_1,\ldots,x_N$.  Euler's identity gives
\[
 f_b=\frac1d\sum_{i=1}^{N}x_i\partial_i f_b.
\]
Hence
\begin{equation}\label{eq:stabilized-residue-field}
 K_b:=\left\{x_1,\ldots,x_N;
 \frac1d\partial_1f_b,\ldots,\frac1d\partial_Nf_b\right\}
\end{equation}
is a graded Koszul matrix factorization.  It is the stabilization of the
residue field $S/(x_1,\ldots,x_N)$; see
\cite[Equation~(2.5)]{PolishchukVaintrobCohFT}.  Let $K_b(j)$ be its grading
shift by the character $\lambda\mapsto\lambda^j$.

\begin{lemma}\label{lem:point-sector-generators}
For $0\le j\le d-2$, put
\[
 \kappa_j:=\operatorname{ch}\bigl(K_b(j+1)\bigr)
 -\operatorname{ch}\bigl(K_b(j)\bigr).
\]
Then
\[
 T_{b,\Q}:=\Span_{\Q}\{\kappa_0,\ldots,\kappa_{d-2}\}
 \subset\HH_0(\mathcal G_b)
\]
has dimension $d-1$ and
\[
 T_{b,\Q}\otimes_{\Q}\C=T_{f_b,\C}
 =\bigoplus_{s=1}^{d-1}\C\mathbf1_s.
\]
In particular,
\[
 T_{b,\Q}\subset
 \operatorname{ch}\bigl(K_0(\mathcal G_b)_{\Q}\bigr)
 \subset\Hdg(\mathcal G_b,\Q).
\]
\end{lemma}

\begin{proof}
The equivariant Chern-character formula is given in
\cite[Theorem~2.6.1(ii)]{PolishchukVaintrobCohFT}; its value for the
stabilized residue field is computed in
\cite[Proposition~4.3.4]{PolishchukVaintrob}.  It computes the component at
$\gamma_s=\zeta^s\ne1$ by the supertrace on the
fiber at the fixed point $0$.  The underlying fiber of $K_b$ is the exterior
algebra of the scalar representation $\C^N$.  Therefore
\begin{equation}\label{eq:point-sector-chern}
 \pr_s\!\left(\operatorname{ch}\bigl(K_b(j)\bigr)\right)
 =c_s\zeta^{sj}\mathbf1_s,
 \qquad
 c_s=\det(1-\zeta^{-s}I_N)=(1-\zeta^{-s})^N\ne0,
\end{equation}
up to a common sign depending only on the convention for the Koszul
differential.  On the identity sector, a grading shift acts by
$\gamma_0^j=1$, so subtraction cancels that component.  Hence
\[
 \pr_s(\kappa_j)=c_s\zeta^{sj}(\zeta^s-1)\mathbf1_s
 \qquad(1\le s\le d-1).
\]
The matrix
\[
 \bigl(\zeta^{sj}\bigr)_{1\le s\le d-1,\;0\le j\le d-2}
\]
is a Vandermonde matrix, with determinant
\[
 \prod_{1\le r<s\le d-1}(\zeta^s-\zeta^r)\ne0.
\]
Multiplying its rows by the nonzero numbers $c_s(\zeta^s-1)$ preserves its
rank.  Thus the $\kappa_j$ form a basis of the point-sector sum after
complexification.  Since every $\kappa_j$ is a difference of algebraic
Chern characters, its rational span is contained in both the Chern image
and the Hodge space \eqref{eq:lin-hodge}.
\end{proof}

\section{Topological K-theory and the Hodge filtration}

\subsection{The rational topological K-group}

We first compute the periodic cyclic dimension from the graded sector
formula and then identify the rational topological rank.

The Jacobian algebra of a squarefree binary form has Hilbert series
$(1+t+\cdots+t^{d-2})^2$ and Milnor number $(d-1)^2$.  By
\eqref{eq:TS-jac},
\begin{equation}\label{eq:milnor-number}
 \mu(f_b)=(d-1)^{2m+2}.
\end{equation}
Let $\zeta=e^{2\pi i/d}$.  A monomial basis of $\Omega_{f_b}$ may be
indexed by tuples $(a_1,\ldots,a_{2m+2})$ with $1\le a_i\le d-1$; scalar
monodromy acts by $\zeta^{a_1+\cdots+a_{2m+2}}$.  Hence, for
$r\not\equiv0\pmod d$, its trace equals
\[
 \left(\sum_{a=1}^{d-1}\zeta^{ra}\right)^{2m+2}
 =(-1)^{2m+2}=1.
\]
The averaging idempotent $d^{-1}\sum_{r=0}^{d-1}T^r$ projects onto the
invariants.  Therefore
\begin{equation}\label{eq:invariant-rank}
 \dim_{\C}\Omega_{f_b}^{T=1}
 =\frac1d\bigl((d-1)^{2m+2}+d-1\bigr).
\end{equation}
By \eqref{eq:graded-hh-decomposition}, the $d-1$ point sectors add $d-1$
dimensions to \eqref{eq:invariant-rank}.  Moreover,
\eqref{eq:graded-hh-degree} is even because $N=2m+2$ is even, and every
point sector has Hochschild degree zero.  Hence
\begin{align}
 \dim_{\C}\HP_0(\mathcal G_b)
 &=\frac{(d-1)^{2m+2}+d-1}{d}+d-1,
 \label{eq:graded-hp-zero}\\
 \HP_1(\mathcal G_b)&=0.
 \label{eq:graded-hp-one}
\end{align}
The Hodge--de Rham spectral sequence degenerates for smooth proper dg
categories \cite{Kaledin}.  To identify the rational lattice, use the
semi-orthogonal decompositions recalled in Section~8.  By
\cite[Proposition~4.32]{Blanc}, the topological Chern character is an
isomorphism after tensoring with $\C$ for $\Perf(X_b)$.  Additivity of
topological
$K$-theory and periodic cyclic homology under semi-orthogonal
decompositions
\cite[Section~3.2, Theorem~3.20 and Example~3.21]{Lin} gives the same statement
for $\mathcal G_b$, because the complementary components are exceptional.
Equation~\eqref{eq:graded-hp-zero} therefore gives
\eqref{eq:main-top-rank}.

\subsection{The complete Hodge-number calculation}

The Hochschild grading in
\cite[Example~2.6.3]{PolishchukVaintrobCohFT}, applied to the identity
sector of \eqref{eq:equivariant-sector-formula}, gives the coefficient
formula
\begin{equation}\label{eq:hodge-coefficient}
 h_{\mathrm{nc}}^{2m-p,p}
 =[t^{(p+1)d-(2m+2)}](1+t+\cdots+t^{d-2})^{2m+2},
 \qquad0\le p\le2m.
\end{equation}
To make the computation explicit, set $N=2m+2$ and
$k=(p+1)d-N$.  Since
\[
 (1+t+\cdots+t^{d-2})^N=(1-t^{d-1})^N(1-t)^{-N},
\]
the coefficient in \eqref{eq:hodge-coefficient} is
\begin{equation}\label{eq:finite-binomial}
 \sum_{j\ge0}(-1)^j
 \binom Nj
 \binom{k-j(d-1)+N-1}{N-1},
\end{equation}
with a binomial coefficient interpreted as zero when its upper entry is
less than $N-1$.  Summing \eqref{eq:hodge-coefficient} over $p$ is exactly
the roots-of-unity projection \eqref{eq:invariant-rank}.  The $d-1$ point
sectors occur only in Hochschild degree zero and therefore add $d-1$ to the
central type-$(0,0)$ dimension of the graded category.

\section{Binary monodromy}

The purpose of this section is to isolate the part of each binary state
space that can contribute rational Hodge classes to the identity sector of
the graded category.  The explicit factorizations of Section~3 span
$A_F\otimes\C=W_0$, but $\Omega_F$ also has binary cyclotomic eigenspaces
$W_r$ for $r\ne0$.  These eigenspaces are distinct from the $d-1$ point
sectors in \eqref{eq:graded-hh-decomposition}: tensor products of the
$W_r$ occur inside the global identity sector when their characters
multiply to one.  The argument below shows that every rational combination
involving a nonzero binary character has infinite braid monodromy.  Since
the very-general Hodge-locus argument of Section~7 forces rational
type-$(0,0)$ classes in the varying identity sector to lie in a
finite-monodromy sub-local system, this calculation excludes precisely the
unwanted tensor summands.  The point sectors are constant and were treated
algebraically in Lemma~\ref{lem:point-sector-generators}.

Three commuting operations enter the calculation and should not be
confused.  First, a fixed Milnor fiber has an order-$d$ scalar automorphism.
Second, moving the roots of $F$ gives Gauss--Manin, or braid, monodromy.
Third, the scalar automorphism becomes the deck transformation of a cyclic
cover.  The scalar action produces the monodromy decomposition, while the
interaction of the other two actions produces the reduced Burau
representation.

\subsection{Cyclotomic sectors}

Let $F\in U_d$ be a squarefree binary form and put
\[
 M_F=\{F=1\}\subset\C^2,
 \qquad \Omega_F=\Jac(F)\,du\wedge dv.
\]
The standard residue identification of the Jacobian state space with
vanishing cohomology is compatible with Gauss--Manin transport and with the
mixed Hodge structure \cite{Steenbrink}.  We therefore pass freely between
$\Omega_F$ and $H^1(M_F,\C)$.

\subsubsection*{Step 1: why there is a monodromy decomposition}

Set $\zeta=e^{2\pi i/d}$.  Homogeneity gives
$F(\zeta u,\zeta v)=F(u,v)$, and hence
\[
 \tau:M_F\longrightarrow M_F,
 \qquad \tau(u,v)=(\zeta u,\zeta v),
 \qquad \tau^d=\id.
\]
The same formula is defined for every $F\in U_d$.  Consequently $\tau$
commutes with parallel transport in the family of Milnor fibers.

Because $x^d-1$ has no repeated root over $\C$, every linear operator whose
$d$th power is the identity is semisimple.  The spectral projectors of
$\tau$ are therefore
\[
 e_r=\frac1d\sum_{j=0}^{d-1}\zeta^{-rj}\tau^j,
 \qquad W_r=e_r\Omega_F=\ker(\tau-\zeta^r),
 \qquad 0\le r\le d-1.
\]
The elementary character identities
$e_re_s=\delta_{rs}e_r$ and $\sum_re_r=1$ give the required decomposition
\begin{equation}\label{eq:sector-decomposition}
 \Omega_F=W_0\oplus W_1\oplus\cdots\oplus W_{d-1}.
\end{equation}
Since $\tau$ commutes with Gauss--Manin transport, every $W_r$ is a complex
sub-local system.  Notice, however, that an individual $W_r$ is generally
defined only over $\Q(\zeta)$; a rational sub-local system contains all
sectors in the Galois orbit of any character that it contains.

\subsubsection*{Step 2: dimensions of the sectors}

The two partial derivatives of a squarefree binary form form a regular
sequence of degree $d-1$.  Hence
\[
 \operatorname{Hilb}_{\Jac(F)}(q)
 =\frac{(1-q^{d-1})^2}{(1-q)^2}
 =(1+q+\cdots+q^{d-2})^2.
\]
If $h_k=\dim_\C\Jac(F)_k$, then
\[
 h_k=\begin{cases}
 k+1,&0\le k\le d-2,\\
 2d-3-k,&d-1\le k\le2d-4.
 \end{cases}
\]
For a homogeneous $g\in\Jac(F)_k$ one has
\[
 \tau^*(g\,du\wedge dv)=\zeta^{k+2}g\,du\wedge dv.
\]
Thus $W_r$ is the sum of the graded pieces for which
$k+2\equiv r\pmod d$.  For $r=0$ only $k=d-2$ occurs.  For
$2\le r\le d-2$, the two degrees $r-2$ and $d+r-2$ contribute
$r-1$ and $d-r-1$ dimensions; the endpoint cases are obtained by omitting
the degree outside the range.  Therefore
\begin{equation}\label{eq:sector-dimensions}
 \dim W_0=d-1,
 \qquad \dim W_r=d-2\quad(1\le r\le d-1).
\end{equation}
Equivalently, the coefficients of the Hilbert series are recorded by pairs
$(a,b)$ with $1\le a,b\le d-1$ and $a+b\equiv r\pmod d$.  For a general
binary form these pairs are a character count, not a preferred monomial
basis of the Jacobian algebra.

\subsubsection*{Step 2a: identification of the identity sector}

We spell out the passage from Lemma~\ref{lem:binary-chern} to the identity
sector.  By \eqref{eq:binary-boundary-bulk},
\[
 \operatorname{ch}(L_j)=q_j(u,v)\,du\wedge dv,
 \qquad \deg q_j=d-2.
\]
Since
\[
 \tau^*u=\zeta u,\qquad \tau^*v=\zeta v,
 \qquad \tau^*(du\wedge dv)=\zeta^2du\wedge dv,
\]
homogeneity of $q_j$ gives
\begin{align}
 \tau^*\operatorname{ch}(L_j)
 &=q_j(\zeta u,\zeta v)\,d(\zeta u)\wedge d(\zeta v)\notag\\
 &=\zeta^{d-2}q_j(u,v)\,\zeta^2du\wedge dv\notag\\
 &=\zeta^d\operatorname{ch}(L_j)
 =\operatorname{ch}(L_j).
 \label{eq:binary-chern-invariant}
\end{align}
Therefore every generator of $A_F$ belongs to the eigenvalue-one sector:
\begin{equation}\label{eq:AF-in-W0}
 A_F\otimes_\Z\C
 =\Span_\C\{\operatorname{ch}(L_1),\ldots,
                    \operatorname{ch}(L_d)\}
 \subseteq\ker(\tau^*-1)=W_0.
\end{equation}
Lemma~\ref{lem:binary-chern} gives exactly one relation among these $d$
generators,
\[
 \sum_{j=1}^d\operatorname{ch}(L_j)=0,
 \qquad
 \dim_\C(A_F\otimes_\Z\C)=d-1.
\]
On the other hand, the congruence $k+2\equiv0\pmod d$ with
$0\le k\le2d-4$ has the unique solution $k=d-2$.  Hence
\[
 W_0=\Jac(F)_{d-2}\,du\wedge dv,
 \qquad
 \dim_\C W_0=h_{d-2}=d-1.
\]
Combining this dimension equality with the inclusion
\eqref{eq:AF-in-W0} yields
\begin{equation}\label{eq:W0-equals-AF}
 \boxed{\;W_0=A_F\otimes_\Z\C.\;}
\end{equation}
Thus the identity sector is generated by the explicit algebraic Chern
characters of Lemma~\ref{lem:binary-chern}; its braid action is finite by
Remark~\ref{rem:finite-AF}.

\subsection{Configuration space and braid monodromy}

\subsubsection*{Step 3: why the braid group acts}

Work on the affine chart on which no root lies at infinity and fix the
leading coefficient.  Then
\[
 F(u,v)=c\prod_{j=1}^d(u-z_jv),
 \qquad z_i\ne z_j.
\]
The unordered root set is a point of
$\Conf_d(\C)/\mathfrak S_d$.  Its fundamental group is the Artin braid
group
\[
 B_d=\left\langle\sigma_1,\ldots,\sigma_{d-1}\ \middle|\
 \begin{aligned}
 \sigma_i\sigma_j&=\sigma_j\sigma_i&&(|i-j|\ge2),\\
 \sigma_i\sigma_{i+1}\sigma_i&=\sigma_{i+1}\sigma_i\sigma_{i+1}
 \end{aligned}\right\rangle;
\]
see \cite[\S 1]{FunarKohno} and
\cite[\S 1.1]{Venkataramana}.  Geometrically, $\sigma_i$ is a positive half-twist
interchanging the adjacent roots $z_i$ and $z_{i+1}$.  With labeled roots
one obtains the pure braid group instead.

A loop in this configuration space gives a locally trivial family of Milnor
fibers.  Parallel transport yields the Gauss--Manin representation
\[
 \rho_F:B_d\longrightarrow\GL(H^1(M_F,\C))
 \simeq\GL(\Omega_F).
\]
This explains the braid action: it is not an additional symmetry chosen on
a single Jacobian algebra, but the holonomy of the state-space local system
as the roots move.  Since the scalar automorphism $\tau$ is defined over the
whole family,
\[
 \rho_F(\beta)\tau=\tau\rho_F(\beta)
 \qquad(\beta\in B_d),
\]
and each sector $W_r$ is braid-stable.

\subsubsection*{Step 4: the cyclic-cover model}

On $v\ne0$, put $x=u/v$ and $y=v^{-1}$.  The equation $F=1$ becomes
\[
 y^d=c\prod_{j=1}^d(x-z_j).
\]
Its smooth projective compactification $C_{\mathbf z}$ is a degree-$d$
cyclic cover of $\PP^1_x$, totally ramified over the $d$ roots.  The
Riemann--Hurwitz formula gives
\[
 g(C_{\mathbf z})=\frac{(d-1)(d-2)}2.
\]
The deck transformation $y\mapsto\zeta y$ corresponds to $\tau^{-1}$.
The Milnor fiber is obtained by deleting the ramification points over the
roots.  The resulting boundary cohomology is the $(d-1)$-dimensional
augmentation representation and belongs to the eigenvalue-$1$ sector.
For $r\ne0$, the sector $W_r$ is therefore the $\zeta^r$-eigenspace of
$H^1(C_{\mathbf z},\C)$, up to replacing $r$ by $d-r$.  This cyclic-cover
realization, including its compatibility with braid monodromy, is established
in \cite[\S\S 7.3--7.7]{Venkataramana}; see also \cite{FunarKohno}.

\subsection{The reduced Burau representation}

\subsubsection*{Step 5: definition from the infinite cyclic cover}

Let $D_{\mathbf z}=\C\setminus\{z_1,\ldots,z_d\}$ and map every positive
meridian around a puncture to $1\in\Z$.  The corresponding infinite cyclic
cover $\widetilde D_{\mathbf z}\to D_{\mathbf z}$ has deck generator $t$.
Its first homology is a free module of rank $d-1$ over
\[
 \Lambda=\Z[t,t^{-1}].
\]
Every braid lifts to this cover and acts $\Lambda$-linearly.  This action is
the reduced Burau representation
\[
 \overline\beta_d:B_d\longrightarrow\GL_{d-1}(\Lambda);
\]
see \cite[\S 1.1]{Venkataramana} and
\cite[\S 2]{FunarKohno}.  This construction also explains the parameter $t$:
it records the deck action, so evaluating $t$ at a root of unity extracts a
deck-eigenspace of the finite cyclic cover.

In the vanishing-chain basis $e_1,\ldots,e_{d-1}$, the convention of
\cite[\S 1.1]{Venkataramana} gives
\[
 \begin{aligned}
 e_i&\longmapsto-te_i,\\
 e_{i-1}&\longmapsto e_{i-1}+te_i&&(i>1),\\
 e_{i+1}&\longmapsto e_{i+1}+e_i&&(i<d-1),\\
 e_j&\longmapsto e_j&&(|j-i|\ge2).
 \end{aligned}
\]
Thus, for an interior index, the only nonidentity block on
$(e_{i-1},e_i,e_{i+1})$ is
\[
 \begin{pmatrix}1&0&0\\ t&-t&1\\0&0&1\end{pmatrix}.
\]
Adjacent block multiplication gives the braid relation, while disjoint
blocks commute.  Appendix~A records the corresponding matrix calculation.

\subsubsection*{Step 6: specialization and the reduced quotient}

Fix $r\ne0$ and put $t_r=\zeta^r$.  The specialized
$(d-1)$-dimensional Burau module contains the vector
\[
 v_{t_r}=\sum_{j=1}^{d-1}(1+t_r+\cdots+t_r^{j-1})e_j.
\]
Substitution in the preceding formulas shows that every $\sigma_i$ fixes
$v_{t_r}$.  Topologically, this is the boundary class killed when the cyclic
cover is compactified.  Consequently
\[
 E_{t_r}=\left(
 \overline\beta_d\otimes_{\Lambda,t\mapsto t_r}\C
 \right)\big/\C v_{t_r},
 \qquad \dim E_{t_r}=d-2,
\]
and the cyclic-cover construction gives a braid-equivariant isomorphism
\[
 W_r\simeq E_{\zeta^r},
\]
possibly after the harmless exchange $r\leftrightarrow d-r$.  This is the
precise meaning of the statement that the braid action on a nonidentity
sector is the reduced Burau representation ``after quotienting the invariant
line.''

\subsection{The invariant Hermitian form, signature, and irreducibility}

\subsubsection*{The invariant Hermitian form and its signature}

Let $V_t=\C^{d-1}$ be the specialized reduced Burau module, and let
$B_i(t)$ denote the matrix of $\sigma_i$.  A Hermitian form is a map
\[
 h_t:V_t\times V_t\longrightarrow\C,\qquad
 h_t(ax,by)=a\overline b\,h_t(x,y),\qquad
 h_t(y,x)=\overline{h_t(x,y)}.
\]
Its Gram matrix $H_t=(h_t(e_j,e_k))$ is
\begin{equation}\label{eq:squier-gram}
 H_t=
 \begin{pmatrix}
  2+t+t^{-1}&-(1+t)&&0\\
  -(1+t^{-1})&2+t+t^{-1}&\ddots&\\
  &\ddots&\ddots&-(1+t)\\
  0&&-(1+t^{-1})&2+t+t^{-1}
 \end{pmatrix}.
\end{equation}
In the normalization of \cite[\S\S 1.1.3, 4.1]{Venkataramana}, the
invariance theorem \cite{Squier} is
\begin{equation}\label{eq:squier-invariance}
 B_i(t)^*H_tB_i(t)=H_t,\qquad
 B_i(t)^*=\overline{B_i(t)}^{\mathsf T}.
\end{equation}
Since $B_i(t)-I$ is supported on the indices $i-1,i,i+1$,
\eqref{eq:squier-invariance} reduces to the corresponding $3\times3$
matrix identity.

Put
\[
 t=e^{2i\theta},\qquad c=\cos\theta,\qquad
 f_j=e^{ij\theta}e_j.
\]
Then
\[
 2+t+t^{-1}=4c^2,\qquad
 -(1+t)=-2ce^{i\theta},\qquad
 -(1+t^{-1})=-2ce^{-i\theta},
\]
and therefore
\begin{equation}\label{eq:squier-real-matrix}
 [h_t]_{(f_j)}=T_c=
 \begin{pmatrix}
  4c^2&-2c&&0\\
  -2c&4c^2&\ddots&\\
  &\ddots&\ddots&-2c\\
  0&&-2c&4c^2
 \end{pmatrix}.
\end{equation}
For
\[
 s_k=\left(\sin\frac{\pi kj}{d}\right)_{j=1}^{d-1},
 \qquad1\le k\le d-1,
\]
the identity
\[
 \sin\frac{\pi k(j-1)}d+\sin\frac{\pi k(j+1)}d
 =2\cos\frac{\pi k}d\sin\frac{\pi kj}d
\]
gives
\begin{equation}\label{eq:squier-eigenvalues}
 T_cs_k=\lambda_ks_k,\qquad
 \lambda_k=4c\left(c-\cos\frac{\pi k}{d}\right).
\end{equation}

Let $t=t_r=\zeta^r$ and $\theta=\pi r/d$.  In the $f$-basis,
\[
 v_{t_r}
 =\sum_{j=1}^{d-1}\frac{1-e^{2ij\theta}}{1-e^{2i\theta}}e_j
 =\frac{e^{-i\theta}}{\sin\theta}
   \sum_{j=1}^{d-1}\sin(j\theta)f_j.
\]
Thus $v_{t_r}$ is proportional to $s_r$, and
\[
 \ker H_{t_r}=\C v_{t_r}\qquad(r\ne d/2).
\]
The radical is
\[
 \operatorname{rad}(h_{t_r})
 =\{x\in V_{t_r}:h_{t_r}(x,y)=0\text{ for all }y\in V_{t_r}\}.
\]
Hence $h_{t_r}$ induces a nondegenerate Hermitian form
$\overline h_r$ on
\[
 E_{t_r}=V_{t_r}/\C v_{t_r}.
\]
Since $k\mapsto\cos(\pi k/d)$ is strictly decreasing,
\begin{equation}\label{eq:squier-signature}
 \operatorname{sign}(\overline h_r)=
 \begin{cases}
  (d-r-1,r-1),&1\le r<d/2,\\
  (r-1,d-r-1),&d/2<r\le d-1,
 \end{cases}
\end{equation}
where the entries count positive and negative eigenvalues.  The same two
integers, up to order, are
identified with the Hodge numbers of the cyclic-cover eigenspace by the
mixed-Hodge description in \cite{Steenbrink}.

Suppose $d$ is even and $r=d/2$, so $t=-1$ and $H_{-1}=0$.  Define
\begin{equation}\label{eq:minus-one-alternating}
 \omega(e_j,e_{j+1})=-1,\qquad
 \omega(e_{j+1},e_j)=1,\qquad
 \omega(e_j,e_k)=0\quad(|j-k|\ne1).
\end{equation}
Then
\[
 \operatorname{rad}(\omega)
 =\C(e_1+e_3+\cdots+e_{d-1})=\C v_{-1},
\]
and $\omega$ induces a nondegenerate alternating form
$\overline\omega$ on $E_{-1}$.  The specialized Burau operators are
\begin{equation}\label{eq:PL-transvection}
 \sigma_i(x)=x+\overline\omega(x,\delta_i)\delta_i
 =(I+N_i)x,\qquad N_i^2=0,\qquad N_i\ne0,
\end{equation}
where $\delta_i$ is the image of $e_i$ in $E_{-1}$.  Such an operator is a
Picard--Lefschetz transvection.  Moreover,
\[
 \sigma_i^q=(I+N_i)^q=I+qN_i\ne I
 \qquad(q\in\Z\setminus\{0\}),
\]
so it has infinite order.

\subsubsection*{Irreducibility of the specialized quotient}

Write $\delta_i$ for the image of $e_i$ in $E_t$.  For $t\ne-1$,
\begin{equation}\label{eq:complex-reflection}
 \sigma_i(x)
 =x-(1+t)
   \frac{\overline h_t(x,\delta_i)}
        {\overline h_t(\delta_i,\delta_i)}\delta_i,
 \qquad
 \overline h_t(\delta_i,\delta_i)\ne0.
\end{equation}
An operator
\[
 R(x)=x+\ell(x)\delta,\qquad
 \operatorname{rank}(R-I)=1,\qquad R\delta\ne\delta,
\]
is a complex reflection; $\C\delta$ is its reflecting direction and
$\ker\ell$ is its fixed hyperplane.  Formula
\eqref{eq:complex-reflection} exhibits $\sigma_i$ as such a reflection.
By \eqref{eq:squier-gram},
\begin{equation}\label{eq:vanishing-chain-pairing}
 \overline h_t(\delta_i,\delta_j)=0\quad(|i-j|\ge2),\qquad
 \overline h_t(\delta_i,\delta_{i+1})=-(1+t)\ne0.
\end{equation}
The images $\delta_1,\ldots,\delta_{d-1}$ span $E_t$.

Let
\[
 0\ne L\subseteq E_t,\qquad \sigma_i(L)=L
 \quad(1\le i\le d-1),
\]
and choose $0\ne x\in L$.  If
\[
 \overline h_t(x,\delta_i)=0\qquad(1\le i\le d-1),
\]
then $x$ is orthogonal to a spanning set.  Nondegeneracy of
$\overline h_t$ would give $x=0$.  Hence there is an $i$ such that
\[
 0\ne\sigma_i(x)-x
 =-(1+t)
   \frac{\overline h_t(x,\delta_i)}
        {\overline h_t(\delta_i,\delta_i)}\delta_i\in L,
\]
and therefore $\delta_i\in L$.  From
\eqref{eq:vanishing-chain-pairing},
\[
 \begin{aligned}
 i<d-1&\Longrightarrow
 0\ne\sigma_{i+1}(\delta_i)-\delta_i
 \in\C^*\delta_{i+1}\cap L,\\
 i>1&\Longrightarrow
 0\ne\sigma_{i-1}(\delta_i)-\delta_i
 \in\C^*\delta_{i-1}\cap L.
 \end{aligned}
\]
Induction in both directions yields
\[
 \delta_1,\ldots,\delta_{d-1}\in L,\qquad
 L=\Span_\C\{\delta_1,\ldots,\delta_{d-1}\}=E_t.
\]

For $t=-1$, use \eqref{eq:PL-transvection} and the nondegenerate form
$\overline\omega$.  Since
\[
 \overline\omega(\delta_i,\delta_{i+1})=-1,\qquad
 \overline\omega(\delta_i,\delta_j)=0\quad(|i-j|\ge2),
\]
the same argument gives
\[
 0\ne L\subseteq E_{-1},\quad
 \sigma_i(L)=L\ \forall i
 \quad\Longrightarrow\quad L=E_{-1}.
\]
Therefore
\[
 W_r\simeq E_{\zeta^r}\quad\Longrightarrow\quad
 W_r\text{ is an irreducible }B_d\text{-module}
 \qquad(1\le r\le d-1);
\]
compare \cite[Proposition 16]{Venkataramana}.

\subsection{The rational finite-monodromy subspace}

We can now combine the decomposition, the Galois action, and the preceding
signature calculation.

\begin{lemma}\label{lem:orbit}
Let $d\ge7$, let $\zeta=e^{2\pi i/d}$, and let $1\le a\le d-1$.  Define
\[
 \mathcal O(a)=
 \left\{b\in\{1,\ldots,d-1\}:\zeta^b=\sigma(\zeta^a)
 \text{ for some }\sigma\in
 \operatorname{Gal}\bigl(\Q(\zeta^a)/\Q\bigr)\right\}.
\]
Then
\[
 \mathcal O(a)\cap\{2,\ldots,d-2\}\ne\varnothing.
\]
\end{lemma}

\begin{proof}
Put
\[
 g=(a,d),\qquad n=\frac d g=\operatorname{ord}(\zeta^a),
 \qquad q=\frac d n=g.
\]
Since $a\not\equiv0\pmod d$, one has $n\ge2$.  The element $\zeta^a$ is a
primitive $n$th root of unity, and therefore
\begin{equation}\label{eq:cyclotomic-orbit-formula}
 \mathcal O(a)
 =\left\{qu:1\le u\le n-1,\ (u,n)=1\right\},
 \qquad
 \#\mathcal O(a)=\varphi(n).
\end{equation}

Assume
\[
 \mathcal O(a)\cap\{2,\ldots,d-2\}=\varnothing.
\]
Then
\[
 \mathcal O(a)\subseteq\{1,d-1\},
 \qquad \varphi(n)\le2.
\]
Write
\[
 n=2^\alpha\prod_{j=1}^s p_j^{e_j},
 \qquad p_j\ge3,
\]
where the $p_j$ are distinct odd primes.  Euler's formula gives
\begin{equation}\label{eq:small-totient}
 \varphi(n)=
 \begin{cases}
  2^{\alpha-1}\displaystyle\prod_{j=1}^s
     p_j^{e_j-1}(p_j-1),&\alpha\ge1,\\[4pt]
  \displaystyle\prod_{j=1}^s p_j^{e_j-1}(p_j-1),&\alpha=0.
 \end{cases}
\end{equation}
The inequality $\varphi(n)\le2$ implies
\[
 \alpha\le2,\qquad p_j-1\le2,\qquad e_j=1.
\]
Hence every odd prime divisor is $3$, and
\[
 n=2^\alpha3^\varepsilon,\qquad
 0\le\alpha\le2,\qquad 0\le\varepsilon\le1.
\]
Since $\varphi(12)=4$ and $n\ge2$, it follows that
\begin{equation}\label{eq:small-totient-classification}
 n\in\{2,3,4,6\}.
\end{equation}

Because $n\mid d$, $d\ge7$, and $n\in\{2,3,4,6\}$, one has $d\ne n$ and
therefore
\[
 q=\frac d n\ge2.
\]
Taking $u=1$ in \eqref{eq:cyclotomic-orbit-formula} gives $q\in\mathcal O(a)$.
Moreover,
\[
 2\le q=\frac d n\le\frac d2\le d-2.
\]
Thus
\[
 q\in\mathcal O(a)\cap\{2,\ldots,d-2\},
\]
contradicting the assumption.
\end{proof}

\begin{proposition}\label{prop:finite-monodromy}
Let $d\ge7$.  The largest rational braid-stable subspace of $\Omega_F$ on
which braid monodromy has finite image is $A_F\otimes\Q$.
\end{proposition}

\begin{proof}
By Remark~\ref{rem:finite-AF} and \eqref{eq:W0-equals-AF},
\[
 A_F\otimes_\Z\C=W_0,\qquad
 \#\rho_{W_0}(B_d)<\infty.
\]
It remains to prove maximality.  Let
\[
 C\subset\Omega_{F,\Q},\qquad
 \rho(\beta)C=C\quad(\beta\in B_d),\qquad
 \#\rho_C(B_d)<\infty.
\]
Since $\rho(\beta)\tau=\tau\rho(\beta)$ and $\tau^d=1$, put
\[
 C^{\tau}:=\sum_{j=0}^{d-1}\tau^jC.
\]
For every $j$,
\[
 \rho_{\tau^jC}(\beta)
 =\tau^j\rho_C(\beta)\tau^{-j},
\]
and hence
\[
 \rho_{C^\tau}(B_d)
 \hookrightarrow
 \prod_{j=0}^{d-1}\rho_{\tau^jC}(B_d).
\]
The group on the right is finite.  Replacing $C$ by $C^\tau$, assume
\[
 \tau C=C.
\]

After complexification, the projectors
\[
 e_r=\frac1d\sum_{j=0}^{d-1}\zeta^{-rj}\tau^j
\]
give
\begin{equation}\label{eq:C-sector-decomposition}
 C_\C=\bigoplus_{r=0}^{d-1}C_r,\qquad
 C_r=e_rC_\C=C_\C\cap W_r.
\end{equation}
Each $C_r$ is braid-stable because
\[
 \rho(\beta)e_r=e_r\rho(\beta).
\]
For $r\ne0$, irreducibility of $W_r$ gives
\begin{equation}\label{eq:C-sector-zero-or-all}
 C_r=0\quad\text{or}\quad C_r=W_r.
\end{equation}

Assume $C_r=W_r$ for some $r\ne0$.  Since $C$ is defined over $\Q$,
\[
 \sigma(C_\C)=C_\C
 \qquad
 \bigl(\sigma\in\operatorname{Gal}(\overline\Q/\Q)\bigr).
\]
Moreover,
\[
 \sigma(W_r)=W_{r'},\qquad
 \sigma(\zeta^r)=\zeta^{r'}.
\]
Thus
\[
 W_{r'}=\sigma(W_r)\subset C_\C
\qquad
 \bigl(r'\in\mathcal O(r)\bigr).
\]
Lemma~\ref{lem:orbit} supplies
\[
 s\in\mathcal O(r)\cap\{2,\ldots,d-2\},
 \qquad W_s\subset C_\C.
\]

Suppose first that $s\ne d/2$.  By
\eqref{eq:squier-signature},
\[
 \operatorname{sign}(\overline h_s)=(p_s,q_s),\qquad
 p_s>0,\qquad q_s>0.
\]
Hence $\overline h_s$ is nondegenerate and indefinite.  Assume, for a
contradiction, that
\[
 \Gamma_s:=\rho_{W_s}(B_d)
\]
is finite.  Choose any positive-definite Hermitian form $g_0$ on $W_s$ and
define
\begin{equation}\label{eq:averaged-positive-form}
 g(x,y):=\frac1{\#\Gamma_s}
 \sum_{\gamma\in\Gamma_s}g_0(\gamma x,\gamma y).
\end{equation}
Then
\[
 g(x,x)>0\quad(x\ne0),\qquad
 g(\gamma_0x,\gamma_0y)=g(x,y)
 \quad(\gamma_0\in\Gamma_s).
\]

Let $W_s^\dagger$ denote the space of conjugate-linear functionals on
$W_s$.  The two nondegenerate forms determine isomorphisms
\[
 \Phi_g,\Phi_h:W_s\longrightarrow W_s^\dagger,\qquad
 \Phi_g(x)=g(x,-),\qquad
 \Phi_h(x)=\overline h_s(x,-).
\]
Define
\begin{equation}\label{eq:defining-A-from-forms}
 A:=\Phi_g^{-1}\Phi_h\in\GL(W_s).
\end{equation}
Equivalently, $A$ is the unique complex-linear operator satisfying
\begin{equation}\label{eq:h-equals-gA}
 \boxed{\ \overline h_s(x,y)=g(Ax,y)\ }
 \qquad(x,y\in W_s).
\end{equation}
The uniqueness follows from
\[
 g(Ax,y)=g(A'x,y)\ \forall y
 \quad\Longrightarrow\quad
 g((A-A')x,y)=0\ \forall y
 \quad\Longrightarrow\quad A=A'.
\]
Hermitian symmetry gives
\[
 g(Ax,y)=\overline h_s(x,y)
 =\overline{\overline h_s(y,x)}
 =\overline{g(Ay,x)}
 =g(x,Ay),
\]
so $A$ is $g$-self-adjoint.

Write $\rho_s=\rho_{W_s}$.  For $\beta\in B_d$, invariance of both forms
gives
\begin{align*}
 g(A\rho_s(\beta)x,\rho_s(\beta)y)
 &=\overline h_s(\rho_s(\beta)x,\rho_s(\beta)y)\\
 &=\overline h_s(x,y)
 =g(Ax,y)\\
 &=g(\rho_s(\beta)Ax,\rho_s(\beta)y).
\end{align*}
Nondegeneracy of $g$ implies
\[
 A\rho_s(\beta)=\rho_s(\beta)A
 \qquad(\beta\in B_d).
\]
Since $W_s$ is an irreducible complex $B_d$-module, Schur's lemma gives
\[
 A=\lambda I,\qquad \lambda\in\C.
\]
The $g$-self-adjointness and invertibility of $A$ give
\[
 \lambda=\overline\lambda,\qquad \lambda\ne0.
\]
Consequently
\[
 \overline h_s(x,y)=\lambda g(x,y),\qquad
 \operatorname{sign}(\overline h_s)=
 \begin{cases}
  (d-2,0),&\lambda>0,\\
  (0,d-2),&\lambda<0,
 \end{cases}
\]
contradicting $p_s,q_s>0$.  Therefore
\[
 \#\rho_{W_s}(B_d)=\infty.
\]

If $s=d/2$, \eqref{eq:PL-transvection} gives
\[
 \rho_{W_s}(\sigma_i)^q=I+qN_i,\qquad
 N_i\ne0,\qquad q\in\Z,
\]
and hence again
\[
 \#\rho_{W_s}(B_d)=\infty.
\]
Both cases contradict $W_s\subset C_\C$ and
$\#\rho_C(B_d)<\infty$.  Thus, by
\eqref{eq:C-sector-zero-or-all},
\[
 C_r=0\qquad(1\le r\le d-1).
\]
Using \eqref{eq:C-sector-decomposition} and
\eqref{eq:W0-equals-AF},
\[
 C_\C\subset W_0=A_F\otimes_\Z\C.
\]
Taking the rational form yields
\[
 C\subset
 (A_F\otimes_\Z\C)\cap\Omega_{F,\Q}
 =A_F\otimes_\Z\Q.
\]
The reverse inclusion has finite braid monodromy by
Remark~\ref{rem:finite-AF}; hence $A_F\otimes\Q$ is maximal.
\end{proof}

\begin{remark}
Proposition~\ref{prop:finite-monodromy} uses only infinitude and
irreducibility.
The stronger arithmeticity theorem \cite{Venkataramana} is
consistent with, but not necessary for, the argument.  Appendix~A gives a
compact matrix verification of the signature and connected-chain steps.
\end{remark}

\section{Very general rational Hodge classes}

Let $\mathbb V_{\Q}$ be the rational local system underlying the
identity-sector summand of the weight-zero noncommutative Hodge structure of
the family $\mathcal G_b$ over $B$.  Its complexification is the
scalar-invariant summand $\Omega_{f_b}^{T=1}$ in
\eqref{eq:graded-hh-decomposition}.  The identity-sector projector is
defined over $\Q$ and therefore splits the rational topological $K$-local
system.  The twisted de Rham model \cite{Efimov} supplies the flat connection
and filtration, and the rational lattice is induced by the topological
realization \cite{Blanc}.  The polarization is the singularity
pairing on vanishing cohomology, whose compatibility with the Hodge
structure is established in \cite{Steenbrink}; Thom--Sebastiani
compatibility follows from \cite{Preygel}.  Thus the Hodge-locus argument
is applied to the polarizable rational variation attached to the identity
sector of the graded category.  The complementary point-sector local system
is constant of type $(0,0)$ and was generated algebraically in
Lemma~\ref{lem:point-sector-generators}.
Write
\[
 \mathcal V=\mathbb V_{\Q}\otimes_{\Q}\mathcal O_B,\qquad
 \nabla:\mathcal V\longrightarrow
 \mathcal V\otimes\Omega_B^1,
\]
and let $F^\bullet\mathcal V$ be its Hodge filtration.  For every
$b\in B$,
\begin{equation}\label{eq:fiberwise-zero-zero}
 \mathbb V_b^{0,0}
 =F_b^0\cap\overline{F_b^0},\qquad
 \Hdg(\mathbb V_{\Q,b})
 =\mathbb V_{\Q,b}\cap F_b^0\cap\overline{F_b^0}.
\end{equation}

\begin{lemma}\label{lem:very-general-hodge}
At a very general point of a polarizable rational variation of Hodge
structure, every rational Hodge class belongs to a rational sub-local-system
of type $(0,0)$.
\end{lemma}

\begin{proof}
Fix $b_0\in B$ and put
\[
 V_{\Q}:=\mathbb V_{\Q,b_0},\qquad
 \rho:\pi_1(B,b_0)\longrightarrow\GL(V_{\Q}).
\]
On the universal cover $p:\widetilde B\to B$, parallel transport gives
\[
 p^*\mathbb V_{\Q}\simeq V_{\Q}\times\widetilde B.
\]
For $v\in V_{\Q}$, let $v_{\tilde b}$ denote the corresponding flat
section and define
\begin{equation}\label{eq:hodge-locus-v}
 Z(v):=
 \left\{\tilde b\in\widetilde B:
 v_{\tilde b}\in
 F_{\tilde b}^0\cap\overline{F_{\tilde b}^0}\right\}.
\end{equation}
Equivalently, if
\[
 \operatorname{pr}_{p,-p,\tilde b}:
 V_{\C}\longrightarrow V_{\tilde b}^{p,-p}
\]
denotes Hodge projection, then
\begin{equation}\label{eq:hodge-locus-equations}
 Z(v)=
 \bigcap_{p\ne0}
 \left\{\tilde b:
 \operatorname{pr}_{p,-p,\tilde b}(v)=0\right\}.
\end{equation}
By \cite[Theorem~1.1]{CDK}, every descended irreducible
component of a locus \eqref{eq:hodge-locus-v} is algebraic.  Since
\[
 \#V_{\Q}\le\aleph_0,
\]
the union of all proper Hodge loci is a countable union
\[
 \mathcal Z=\bigcup_{\substack{v\in V_{\Q}\\Z(v)\ne\widetilde B}}
 p\bigl(Z(v)\bigr).
\]
Choose
\[
 b\in B\setminus\mathcal Z,\qquad
 \alpha\in
 \mathbb V_{\Q,b}\cap F_b^0\cap\overline{F_b^0}.
\]
After identifying $\mathbb V_{\Q,b}$ with $V_{\Q}$ along a path, one has
$b\in p(Z(\alpha))$.  Hence
\[
 Z(\alpha)=\widetilde B.
\]
For every $\gamma\in\pi_1(B,b)$,
\[
 Z\bigl(\rho(\gamma)\alpha\bigr)=\widetilde B.
\]
Define
\begin{equation}\label{eq:monodromy-orbit-system}
 L_{\Q,b}:=
 \Span_{\Q}\{\rho(\gamma)\alpha:
 \gamma\in\pi_1(B,b)\}\subset\mathbb V_{\Q,b}.
\end{equation}
Then
\[
 \rho(\delta)L_{\Q,b}=L_{\Q,b}
 \qquad(\delta\in\pi_1(B,b)),
\]
so $L_{\Q,b}$ is the fiber of a rational sub-local-system
$\mathbb L_{\Q}\subset\mathbb V_{\Q}$.  Equations
\eqref{eq:hodge-locus-equations} and
\eqref{eq:monodromy-orbit-system} give
\[
 \mathbb L_{\C,\tilde b}
 \subset F_{\tilde b}^0\cap\overline{F_{\tilde b}^0}
 =\mathbb V_{\tilde b}^{0,0}
 \qquad(\tilde b\in\widetilde B).
\]
\end{proof}

\begin{lemma}\label{lem:type-zero-finite}
An integral polarizable local system of pure type $(0,0)$ has finite
monodromy.
\end{lemma}

\begin{proof}
Let $\mathbb L_{\Z}$ be the integral lattice and let
\[
 Q:\mathbb L_{\Z}\times\mathbb L_{\Z}\longrightarrow\Z
\]
be the polarization.  Since
\[
 \mathbb L_{\C}=\mathbb L^{0,0},
\]
the Hodge--Riemann relation is
\[
 Q(x,\overline x)>0\qquad(0\ne x\in\mathbb L_{\C}).
\]
Thus, in an integral basis of $L_{\Z}:=\mathbb L_{\Z,b}$,
\[
 Q=(q_{ij})\in M_n(\Z),\qquad
 Q^{\mathsf T}=Q,\qquad Q>0.
\]
Monodromy satisfies
\begin{equation}\label{eq:integral-orthogonal-monodromy}
 \rho(\gamma)\in
 O(Q,\Z):=
 \{M\in\GL_n(\Z):M^{\mathsf T}QM=Q\}.
\end{equation}
Let $\mu_{\min}>0$ be the smallest eigenvalue of $Q$.  If
$M=(m_1,\ldots,m_n)\in O(Q,\Z)$, then
\[
 m_j^{\mathsf T}Qm_j=q_{jj},\qquad
 \mu_{\min}\|m_j\|^2\le q_{jj}.
\]
Hence
\[
 m_j\in
 \left\{m\in\Z^n:
 \|m\|^2\le\frac{\max_iq_{ii}}{\mu_{\min}}\right\},
\]
a finite set.  Therefore
\[
 \#O(Q,\Z)<\infty,
 \qquad
 \#\rho\bigl(\pi_1(B,b)\bigr)<\infty.
\]
\end{proof}

\begin{proposition}\label{prop:Hodge-equals-A}
For very general $b\in B$,
\[
 \Hdg_{\mathrm{id}}(\mathcal G_b,\Q)
 :=\Hdg(\mathcal G_b,\Q)\cap(\Omega_{f_b})_{d(m+1)}
 =A_{b,\Q}.
\]
\end{proposition}

\begin{proof}
We prove the two inclusions separately.
Proposition~\ref{prop:algebraic-lattice} gives
\begin{equation}\label{eq:A-contained-Hdg}
 A_{b,\Q}\subseteq\Hdg_{\mathrm{id}}(\mathcal G_b,\Q).
\end{equation}
Let
\[
 \alpha\in\Hdg_{\mathrm{id}}(\mathcal G_b,\Q),
 \qquad b\in B\setminus\mathcal Z.
\]
We must prove $\alpha\in A_{b,\Q}$.  At a very general parameter, the
parallel transports of $\alpha$ remain of type $(0,0)$.  They generate the
monodromy-orbit space
\[
 L_{\Q,b}
 =\Span_{\Q}\{\rho(\gamma)\alpha:
                 \gamma\in\pi_1(B,b)\}.
\]
Lemma~\ref{lem:very-general-hodge} says precisely that this space is the
fiber of a rational sub-local-system $\mathbb L_{\Q}$:
\[
 \alpha\in L_{\Q,b},\qquad
 \mathbb L_{\Q}\subset\mathbb V_{\Q},\qquad
 \mathbb L_{\C,b}\subset\mathbb V_b^{0,0}.
\]
To apply the polarization, retain the integral vectors in this rational
system:
\[
 \mathbb L_{\Z}:=
 \mathbb L_{\Q}\cap\mathbb V_{\Z}.
\]
The restriction of the polarization to a type-$(0,0)$ system is positive
definite.  Lemma~\ref{lem:type-zero-finite} therefore gives
\begin{equation}\label{eq:L-finite-monodromy}
 \#\rho_{\mathbb L}\bigl(\pi_1(B,b)\bigr)<\infty.
\end{equation}

Write
\[
 B=\prod_{\nu=0}^{m}U_d^{(\nu)},\qquad
 G:=\pi_1(B,b)=\prod_{\nu=0}^{m}B_d^{(\nu)}.
\]
The product decomposition records that the $m+1$ binary forms vary
independently.  Accordingly, the factor $B_d^{(\nu)}$ moves only the roots
of $F_\nu$ and acts trivially on the other binary state spaces.

For each binary factor, $\tau_\nu$ is its order-$d$ scalar monodromy.
The operator $e_{\nu,r}$ is the spectral projector onto the
$\zeta^r$-eigenspace.  Explicitly,
\[
 \tau_\nu^d=1,\qquad
 e_{\nu,r}:=
 \frac1d\sum_{j=0}^{d-1}\zeta^{-rj}\tau_\nu^j,
 \qquad
 W_{\nu,r}:=e_{\nu,r}\Omega_{F_\nu}.
\]
The operators $\tau_0,\ldots,\tau_m$ commute with one another and with
$G$.  The subspace $\mathbb L_{\C,b}$ need not itself be stable under every
$\tau_\nu$, so the projectors $e_{\nu,r}$ cannot yet be applied internally.
For this reason, take the smallest subspace containing
$\mathbb L_{\C,b}$ and stable under all scalar monodromies:
\begin{equation}\label{eq:L-scalar-closure}
 \widetilde L_{\C}:=
 \sum_{0\le a_0,\ldots,a_m\le d-1}
 \tau_0^{a_0}\cdots\tau_m^{a_m}\mathbb L_{\C,b}.
\end{equation}
This is the scalar closure of $\mathbb L_{\C,b}$.  It does not destroy
finite braid monodromy: the scalar operators commute with $G$, so the
$G$-representation on every summand in
\eqref{eq:L-scalar-closure} is conjugate to the representation on
$\mathbb L_{\C,b}$.  Indeed, \eqref{eq:L-finite-monodromy} and
$\tau_\nu\rho(\gamma)=\rho(\gamma)\tau_\nu$ imply
\[
 \rho_{\widetilde L}(G)
 \hookrightarrow
 \prod_{0\le a_0,\ldots,a_m\le d-1}
 \rho_{\tau_0^{a_0}\cdots\tau_m^{a_m}\mathbb L}(G),
\]
where every group on the right is finite.  Hence
\[
 \#\rho_{\widetilde L}(G)<\infty.
\]

The multi-index $\mathbf r=(r_0,\ldots,r_m)$ records one scalar eigenvalue
for each binary factor.  The product projector
$e_{\mathbf r}$ extracts the simultaneous eigenspace on which
$\tau_\nu$ acts by $\zeta^{r_\nu}$:
\[
 e_{\mathbf r}:=\prod_{\nu=0}^{m}e_{\nu,r_\nu},
 \qquad
 \tau_\nu e_{\mathbf r}
 =\zeta^{r_\nu}e_{\mathbf r}.
\]
Thus, for
\[
 \mathbf r=(r_0,\ldots,r_m)\in(\Z/d\Z)^{m+1},
\]
Thom--Sebastiani and \eqref{eq:sector-decomposition} give
\begin{equation}\label{eq:TS-multisectors}
 \mathbb V_{\C,b}
 =
 \bigoplus_{\substack{\mathbf r\in(\Z/d\Z)^{m+1}\\
 r_0+\cdots+r_m\equiv0\;(\mathrm{mod}\ d)}}
 U_{\mathbf r},
 \qquad
 U_{\mathbf r}:=
 \bigotimes_{\nu=0}^{m}W_{\nu,r_\nu}.
\end{equation}
The congruence
\[
 r_0+\cdots+r_m\equiv0\pmod d
\]
is the condition that the total scalar monodromy on the tensor product is
the identity; precisely these multi-sectors form the identity orbifold
sector of $\mathcal G_b$.  Since $\widetilde L_{\C}$ is stable under every $\tau_\nu$, it
is stable under every $e_{\mathbf r}$ and therefore decomposes as
\begin{equation}\label{eq:L-multisector-decomposition}
 \widetilde L_{\C}
 =\bigoplus_{\mathbf r}\widetilde L_{\mathbf r},
 \qquad
 \widetilde L_{\mathbf r}
 :=e_{\mathbf r}\widetilde L_{\C}
 =\widetilde L_{\C}\cap U_{\mathbf r}.
\end{equation}

Recall that a $B_d^{(\nu)}$-module is irreducible if its only
$B_d^{(\nu)}$-stable subspaces are $0$ and the whole module.  For
$r_\nu\ne0$, the irreducibility of $W_{\nu,r_\nu}$ was proved in
Section~6.4.  For $r_\nu=0$, \eqref{eq:W0-equals-AF} identifies
\[
 W_{\nu,0}\simeq
 \left\{(x_1,\ldots,x_d)\in\C^d:
 \sum_{i=1}^dx_i=0\right\},
\]
with the standard irreducible representation of $\mathfrak S_d$, through
which $B_d^{(\nu)}$ acts.  Hence every $W_{\nu,r_\nu}$ is irreducible and
$U_{\mathbf r}$ is irreducible under
$G=\prod_\nu B_d^{(\nu)}$.  The projected space
$\widetilde L_{\mathbf r}$ is $G$-stable because both
$\widetilde L_{\C}$ and $e_{\mathbf r}$ commute with $G$.  Irreducibility
therefore gives the all-or-nothing alternative
\begin{equation}\label{eq:L-multisector-zero-all}
 \widetilde L_{\mathbf r}=0
 \quad\text{or}\quad
 \widetilde L_{\mathbf r}=U_{\mathbf r}.
\end{equation}
Assume
\[
 \widetilde L_{\mathbf r}=U_{\mathbf r}\ne0,
 \qquad r_\nu\ne0
\]
for some $\nu$.  We now use the independence of the binary parameter
spaces.  Restricting the $G$-action to $B_d^{(\nu)}$ means that only the
roots of $F_\nu$ move; every other tensor factor remains fixed.  Therefore
\begin{equation}\label{eq:one-factor-tensor-action}
 \rho_{U_{\mathbf r}}(\beta)
 =
 I\otimes\cdots\otimes
 \rho_{W_{\nu,r_\nu}}(\beta)
 \otimes\cdots\otimes I.
\end{equation}
If $T\in\GL(W_{\nu,r_\nu})$ and the other tensor factors are nonzero, then
\[
 I\otimes\cdots\otimes T\otimes\cdots\otimes I=I
 \quad\Longrightarrow\quad T=I.
\]
Thus tensoring with the identity operators does not kill any nontrivial
operator on $W_{\nu,r_\nu}$: the homomorphism
\[
 \GL(W_{\nu,r_\nu})
 \longrightarrow\GL(U_{\mathbf r}),\qquad
 T\longmapsto
 I\otimes\cdots\otimes T\otimes\cdots\otimes I
\]
is injective.  Section~6.4 and Proposition~\ref{prop:finite-monodromy} give
\[
 \#\rho_{W_{\nu,r_\nu}}\bigl(B_d^{(\nu)}\bigr)=\infty,
\]
and hence, by \eqref{eq:one-factor-tensor-action},
\[
 \#\rho_{U_{\mathbf r}}\bigl(B_d^{(\nu)}\bigr)=\infty.
\]
This contradicts $\#\rho_{\widetilde L}(G)<\infty$.  Consequently,
\[
 \widetilde L_{\mathbf r}\ne0
 \quad\Longrightarrow\quad
 \mathbf r=(0,\ldots,0).
\]
In other words, finite monodromy removes every simultaneous eigenspace in
which at least one scalar monodromy has a nontrivial eigenvalue.  The only
remaining component is the identity multi-sector.  Therefore
equations \eqref{eq:TS-multisectors} and
\eqref{eq:L-multisector-decomposition} yield
\[
 \mathbb L_{\C,b}
 \subseteq\widetilde L_{\C}
 \subseteq U_{\mathbf0}
 =\bigotimes_{\nu=0}^{m}W_{\nu,0}.
\]
The identity sector $W_{\nu,0}$ of each binary factor was identified in
\eqref{eq:W0-equals-AF} with the complex span of the explicit Chern
characters.  Hence
\[
 U_{\mathbf0}
 =\bigotimes_{\nu=0}^{m}
 \bigl(A_{F_\nu}\otimes_\Z\C\bigr)
 =A_b\otimes_\Z\C.
\]
Since $\alpha$ is rational, membership in this complexification is
equivalent to membership in its rational form.  Therefore
\[
 \alpha\in
 (A_b\otimes_\Z\C)\cap\mathbb V_{\Q,b}
 =A_{b,\Q}.
\]
Together with \eqref{eq:A-contained-Hdg}, this proves
\[
 \Hdg_{\mathrm{id}}(\mathcal G_b,\Q)=A_{b,\Q}.
\]
\end{proof}

\begin{proof}[Proof of Theorem~\ref{thm:main-graded}]
The rational sector decomposition is
\[
 \HH_0(\mathcal G_b)_{\Q}
 =(\Omega_{f_b})_{d(m+1),\Q}\oplus T_{b,\Q}.
\]
Both summands are compatible with the negative-cyclic filtration and the
topological rational lattice.  Since $T_{b,\Q}$ is constant of type
$(0,0)$, Lemma~\ref{lem:point-sector-generators} and
Proposition~\ref{prop:Hodge-equals-A} give
\[
 \Hdg(\mathcal G_b,\Q)
 =\Hdg_{\mathrm{id}}(\mathcal G_b,\Q)\oplus T_{b,\Q}
 =A_{b,\Q}\oplus T_{b,\Q}
 \subseteq\operatorname{ch}\bigl(K_0(\mathcal G_b)_{\Q}\bigr).
\]
The reverse inclusion follows from compatibility of the algebraic,
negative-cyclic, and topological Chern characters.  Hence
\[
 \operatorname{ch}\bigl(K_0(\mathcal G_b)_{\Q}\bigr)
 =\Hdg(\mathcal G_b,\Q)
 =A_{b,\Q}\oplus T_{b,\Q}.
\]
The Hodge rank is $(d-1)^{m+1}+d-1$, and the topological ranks were computed
in \eqref{eq:graded-hp-zero}--\eqref{eq:graded-hp-one}.  This proves all
assertions of Theorem~\ref{thm:main-graded}.
\end{proof}

\section{Projective hypersurfaces}

We apply the additivity theorem
\cite[Section~3.1, Theorem~3.13 and Remark~3.14]{Lin} to the description of the graded
matrix-factorization category in \cite[Theorem~3.11]{Orlov}.  Put $N=2m+2$.

\begin{theorem}\label{thm:orlov-decompositions}
Let $f\in\C[x_1,\ldots,x_N]$ be homogeneous of degree $d$, assume that
$X=\{f=0\}\subset\PP^{N-1}$ is smooth, and put
$\mathcal G_f=\MFgr(f)$.  By \cite[Theorem~3.11]{Orlov}, there are
dg-enhanced semi-orthogonal decompositions
\begin{align}
 d<N:\quad
 \Perf(X)&=\bigl\langle
 \mathcal O_X(d-N+1),\ldots,\mathcal O_X,\mathcal G_f
 \bigr\rangle,
 \label{eq:orlov-fano}\\
 d=N:\quad
 \Perf(X)&\simeq\mathcal G_f,
 \label{eq:orlov-cy}\\
 d>N:\quad
 \mathcal G_f&=\bigl\langle
 K_f,K_f(-1),\ldots,K_f(N-d+1),\Perf(X)
 \bigr\rangle,
 \label{eq:orlov-general-type}
\end{align}
where $K_f(j)$ is a grading shift of the stabilized residue field.  The line
bundles in \eqref{eq:orlov-fano} and the objects preceding $\Perf(X)$ in
\eqref{eq:orlov-general-type} are exceptional.
\end{theorem}

\begin{proof}
The semi-orthogonal decompositions are those of
\cite[Theorem~3.11]{Orlov}; the dg-enhanced form follows from the
equivariant factorization formulation in \cite[Theorem~6.13]{BFK}.  The
first exceptional sequence has $N-d$ terms, and the third has $d-N$ terms.
\end{proof}

\begin{proof}[Proof of Corollary~\ref{cor:projective}]
If $\nabla f_b=0$, then disjointness of the variables implies
$\nabla F_s=0$ for every $s$.  A squarefree binary form has no nonzero
common zero of its partial derivatives, since such a point would be a
multiple projective root.  Thus the only critical point of $f_b$ is the
origin, and $X_b$ is smooth.

By Theorem~\ref{thm:main-graded}, the noncommutative Hodge conjecture holds
for $\mathcal G_b$.  Apply
\cite[Section~3.1, Theorem~3.13 and Remark~3.14]{Lin} to the
decompositions in Theorem~\ref{thm:orlov-decompositions}.  If $d<N$, then
\eqref{eq:orlov-fano} writes $\Perf(X_b)$ as the semi-orthogonal sum of
$\mathcal G_b$ and $N-d$ exceptional components.  If $d=N$, then
\eqref{eq:orlov-cy} is an equivalence.  If $d>N$, then
\eqref{eq:orlov-general-type} writes $\mathcal G_b$ as the
semi-orthogonal sum of $\Perf(X_b)$ and $d-N$ exceptional components.  The
conjecture holds for an exceptional component because it is Morita
equivalent to $\Perf(\C)$ and its Hodge space is generated by the Chern
character of its exceptional generator.  Therefore
\[
 \operatorname{NCHC}\bigl(\Perf(X_b)\bigr)
\]
holds in all three cases.  The comparison with the classical cycle-class
formulation in \cite[Section~3.1, Remark~3.9]{Lin} now gives the rational Hodge
conjecture for $X_b$.

The same additivity theorem, together with
\eqref{eq:main-hodge-rank} and
\eqref{eq:orlov-fano}--\eqref{eq:orlov-general-type}, gives
\[
 \dim_{\Q}\Hdg\bigl(\Perf(X_b),\Q\bigr)
 =\begin{cases}
 \dim_{\Q}\Hdg(\mathcal G_b,\Q)+(N-d),&d<N,\\
 \dim_{\Q}\Hdg(\mathcal G_b,\Q),&d=N,\\
 \dim_{\Q}\Hdg(\mathcal G_b,\Q)-(d-N),&d>N.
 \end{cases}
\]
In every case, this dimension is
\[
 \dim_{\Q}\Hdg\bigl(\Perf(X_b),\Q\bigr)
 =(d-1)^{m+1}+N-1.
\]
Weak Lefschetz \cite[Chapter~7]{Voisin} gives exactly $N-1$ ambient
rational Hodge classes, one in each even degree, including $h^m$ in the
middle.  Therefore
\[
 \dim_{\Q}\Hdg^m_{\mathrm{prim}}(X_b)=(d-1)^{m+1},
 \qquad
 \dim_{\Q}\Hdg^m(X_b)=(d-1)^{m+1}+1.
\]
\end{proof}

\section{Non-Fermat examples}\label{sec:nonfermat}

We prove in this section that the hypersurfaces constructed above are not
Fermat hypersurfaces.  More precisely,
Proposition~\ref{prop:nonfermat} shows that a general member of the split
family is not linearly equivalent to the Fermat hypersurface of the same
degree and dimension.  This verifies that the main theorems supply
genuinely new examples rather than coordinate transforms of the classical
Fermat cases.

Choose
\begin{equation}\label{eq:explicit-split}
 F_s(u_s,v_s)=u_sv_s\prod_{r=1}^{d-2}(u_s-a_{sr}v_s),
\end{equation}
with all roots distinct in each block and the full parameter outside the
countable exceptional union in Lemma~\ref{lem:very-general-hodge}.  These
are actual smooth split hypersurfaces.

\begin{proposition}\label{prop:nonfermat}
A general polynomial of the form \eqref{eq:split-polynomial} is not linearly
equivalent to the Fermat polynomial.
\end{proposition}

\begin{proof}
The Hessian matrix of $f_b$ is block diagonal, so
\[
 \det\Hess(f_b)=\prod_{s=0}^{m}\det\Hess(F_s).
\]
For a general binary form, its binary Hessian has $2(d-2)$ distinct roots.
Reducedness is an open condition on the coefficients; for background on
Hessian hypersurfaces, see \cite[Chapter~1, \S~1.1.4]{Dolgachev}.  Thus the
Hessian divisor of $f_b$ is a union
of $2(m+1)(d-2)$ distinct hyperplanes, each of multiplicity one.  The
Hessian determinant of $x_0^d+\cdots+x_{2m+1}^d$ is a nonzero scalar
multiple of $\prod_i x_i^{d-2}$, which has only $2m+2$ hyperplanes, each of
multiplicity $d-2$.  The component multiplicities of the Hessian divisor
are invariant under linear coordinate change.
\end{proof}

\section{Exact numerical invariants}

We now record, for the same examples, the exact ranks appearing in
Theorem~\ref{thm:main-graded} and Corollary~\ref{cor:projective}.  This
comparison exhibits the difference
between the entire rational topological $K$-group and its rational Hodge
subspace, separates the identity and point sectors, and distinguishes the
primitive middle Hodge rank from the full middle Hodge rank.

Set
\[
 \begin{aligned}
 P_{d,m}&=\frac{(d-1)^{2m+2}+d-1}{d},
 &G_{d,m}&=P_{d,m}+d-1,\\
 A_{d,m}&=(d-1)^{m+1},
 &H_{d,m}&=A_{d,m}+d-1.
 \end{aligned}
\]
Here $P_{d,m}$ is the dimension of the identity part of periodic cyclic
homology, $G_{d,m}$ is
$\rk K_0^{\mathrm{top}}(\mathcal G_b)_{\Q}$, and $H_{d,m}$ is
$\dim_{\Q}\Hdg(\mathcal G_b,\Q)$.  Applying
\cite[Section~3.1, Theorem~3.13 and Remark~3.14]{Lin} to the decompositions of
\cite[Theorem~3.11]{Orlov}, as in the
proof of Corollary~\ref{cor:projective}, identifies $A_{d,m}$ with the
primitive middle rational Hodge rank of $X_b$.  The full middle rank is
$A_{d,m}+1$.

\begin{center}
\begin{tabular}{@{}cccrrrr@{}}
\toprule
$\dim X$&$d$&type&$G_{d,m}$&$H_{d,m}$&$A_{d,m}$&full middle\\
\midrule
$4$&$8$&general&$14{,}714$&$350$&$343$&$344$\\
$6$&$7$&Fano&$239{,}952$&$1{,}302$&$1{,}296$&$1{,}297$\\
$6$&$8$&Calabi--Yau&$720{,}608$&$2{,}408$&$2{,}401$&$2{,}402$\\
$8$&$9$&Fano&$119{,}304{,}656$&$32{,}776$&$32{,}768$&$32{,}769$\\
\bottomrule
\end{tabular}
\end{center}

For the octic fourfold, the graded Hochschild formula
\cite[Example~2.6.3]{PolishchukVaintrobCohFT}, together with
\eqref{eq:hodge-coefficient} and the decomposition
\cite[Theorem~3.11]{Orlov}, gives
\begin{equation}\label{eq:octic-vector}
 (h_{\mathrm{prim}}^{4,0},h_{\mathrm{prim}}^{3,1},
 h_{\mathrm{prim}}^{2,2},h_{\mathrm{prim}}^{1,3},
 h_{\mathrm{prim}}^{0,4})
 =(21,2667,9331,2667,21).
\end{equation}
The primitive rational Hodge subspace inside the
$9{,}331$-dimensional complex $(2,2)$ summand has rank $343$; adding
$h^2$ gives full rational rank $344$.

For the septic sixfold the primitive Hodge vector is
\begin{equation}\label{eq:septic-vector}
 (0,1708,50288,135954,50288,1708,0),
\end{equation}
and the primitive rational middle Hodge rank is $1296$.  For the nonic
eightfold it is
\begin{equation}\label{eq:nonic-vector}
 (0,24300,2638800,27889620,58199208,
 27889620,2638800,24300,0),
\end{equation}
and the primitive rational middle Hodge rank is $32768$.

\begin{remark}
These primitive classes are not generated by divisor classes.  Weak
Lefschetz gives $H^2(X_b,\Q)=\Q h$ for a smooth hypersurface of dimension
at least four \cite[Chapter 7]{Voisin}.  Products of rational $(1,1)$
classes span only $\Q h^m$, the nonprimitive middle summand.  The
$A_{d,m}$ primitive classes produced by matrix factorizations are additional
codimension-$m$ classes.
\end{remark}

\appendix
\section{The reduced-Burau calculation}

We record the matrix facts used in Proposition~\ref{prop:finite-monodromy}.
Put $t=e^{2\pi ir/d}\ne1$.  After one branch point is fixed, the nontrivial
eigenspace has a vanishing basis $\delta_1,\ldots,\delta_{d-2}$.  The
standard braid generator $T_i$ acts as a complex reflection
\begin{equation}\label{eq:appendix-reflection}
 T_i(v)=v-(1+t)\frac{h_t(v,\delta_i)}{h_t(\delta_i,\delta_i)}\delta_i,
\end{equation}
with the usual limiting transvection interpretation when the denominator
vanishes.  Replacing $t$ by $t^{-1}$ changes conventions but not the
argument.

The Gram matrix of $h_t$ is tridiagonal.  Its diagonal entries are nonzero,
its adjacent entries are nonzero, and all other entries vanish.  Thus its
incidence graph is the connected chain
\[
 \delta_1-\delta_2-\cdots-\delta_{d-2}.
\]
Diagonalizing the associated Toeplitz form by the sine vectors
\[
 \left(\sin\frac{\pi kj}{d-1}\right)_{j=1}^{d-2},
 \qquad 1\le k\le d-2,
\]
gives signature $(r-1,d-r-1)$, up to order.  This agrees with the
Hodge-number formula for the cyclic-cover eigenspace arising from the mixed
Hodge structure \cite{Steenbrink}; the invariant Burau form itself is the
one of \cite{Squier}.

For completeness, let $W\ne0$ be invariant under all $T_i$ and take
$0\ne v\in W$.  If $h_t(v,\delta_i)=0$ for every $i$, nondegeneracy gives
$v=0$, a contradiction.  Therefore \eqref{eq:appendix-reflection} shows
that $W$ contains some $\delta_i$.  Applying the adjacent reflections and
using the nonzero neighboring Gram entries gives every $\delta_j$.  Hence
$W$ is the whole eigenspace.  This proves irreducibility.  If the signature
is indefinite, a finite image would preserve both $h_t$ and an averaged
positive-definite form, contradicting Schur's lemma.  At $t=-1$, one of the
limiting transvections is nontrivial and unipotent, so its powers are
distinct.  These are exactly the assertions used in
Proposition~\ref{prop:finite-monodromy}.

\section{Arithmetic checks}

The numerical rows above follow without floating-point approximation.  For
$m=2,d=8$,
\[
 \begin{aligned}
 P_{8,2}&=\frac{7^6+7}{8}=14{,}707,
 &G_{8,2}&=14{,}714,\\
 A_{8,2}&=343,
 &H_{8,2}&=350.
 \end{aligned}
\]
For $m=3$,
\[
 \begin{aligned}
 P_{7,3}&=\frac{6^8+6}{7}=239{,}946,
 &G_{7,3}&=239{,}952,\\
 A_{7,3}&=1{,}296,
 &H_{7,3}&=1{,}302,
 \end{aligned}
\]
and
\[
 \begin{aligned}
 P_{8,3}&=\frac{7^8+7}{8}=720{,}601,
 &G_{8,3}&=720{,}608,\\
 A_{8,3}&=2{,}401,
 &H_{8,3}&=2{,}408.
 \end{aligned}
\]
For $m=4,d=9$,
\[
 \begin{aligned}
 P_{9,4}&=\frac{8^{10}+8}{9}=119{,}304{,}648,
 &G_{9,4}&=119{,}304{,}656,\\
 A_{9,4}&=32{,}768,
 &H_{9,4}&=32{,}776.
 \end{aligned}
\]
Adding $d-1$ to $P_{d,m}$ and $A_{d,m}$ gives the graded topological and
Hodge ranks, respectively.  Adding one to $A_{d,m}$ gives the full middle
rational Hodge rank.  The vectors
\eqref{eq:octic-vector}--\eqref{eq:nonic-vector} follow from the finite
binomial sum \eqref{eq:finite-binomial}.

\section*{AI disclosure}

All mathematical ideas in this article originate with the authors, and the
work was inspired by the two articles on the noncommutative Hodge conjecture
by Xun Lin and Alexander Perry, respectively, \cite{Lin,Perry}.  In
interaction with the authors, GPT-5.6 Pro assisted with verification,
calculations, and certain technical proofs, reviewed the manuscript and
suggested necessary corrections, and polished the language.  The main body
was written by the authors, who bear sole responsibility for the correctness
of all statements and proofs.


\begin{thebibliography}{PV16}

\bibitem[Aok87]{Aoki}
N.~Aoki,
\emph{Some new algebraic cycles on Fermat varieties},
J. Math. Soc. Japan \textbf{39} (1987), no.~3, 385--396.

\bibitem[BFK14]{BFK}
M.~Ballard, D.~Favero, and L.~Katzarkov,
\emph{A category of kernels for equivariant factorizations and its
implications for Hodge theory},
Publ. Math. Inst. Hautes \`Etudes Sci. \textbf{120} (2014), 1--111.

\bibitem[Bla16]{Blanc}
A.~Blanc,
\emph{Topological $K$-theory of complex noncommutative spaces},
Compos. Math. \textbf{152} (2016), no.~3, 489--555.

\bibitem[CDK95]{CDK}
E.~Cattani, P.~Deligne, and A.~Kaplan,
\emph{On the locus of Hodge classes},
J. Amer. Math. Soc. \textbf{8} (1995), no.~2, 483--506.

\bibitem[Dol12]{Dolgachev}
I.~V.~Dolgachev,
\emph{Classical Algebraic Geometry: A Modern View},
Cambridge University Press, Cambridge, 2012.

\bibitem[Dyc11]{Dyckerhoff}
T.~Dyckerhoff,
\emph{Compact generators in categories of matrix factorizations},
Duke Math. J. \textbf{159} (2011), no.~2, 223--274.

\bibitem[Efi18]{Efimov}
A.~I.~Efimov,
\emph{Cyclic homology of categories of matrix factorizations},
Int. Math. Res. Not. IMRN (2018), no.~12, 3834--3869.

\bibitem[FK14]{FunarKohno}
L.~Funar and T.~Kohno,
\emph{On Burau representations at roots of unity},
Geom. Dedicata \textbf{169} (2014), 145--163.

\bibitem[Kal17]{Kaledin}
D.~Kaledin,
\emph{Spectral sequences for cyclic homology},
in \emph{Algebra, Geometry, and Physics in the 21st Century},
Progr. Math., vol.~324, Birkh\"auser/Springer, Cham, 2017, 99--129.

\bibitem[Lin23]{Lin}
X.~Lin,
\emph{Noncommutative Hodge conjecture},
J. Noncommut. Geom. \textbf{17} (2023), no.~4, 1367--1389.

\bibitem[Orl09]{Orlov}
D.~Orlov,
\emph{Derived categories of coherent sheaves and triangulated categories
of singularities},
in \emph{Algebra, Arithmetic, and Geometry: In Honor of Yu.~I.~Manin},
vol.~II, Progr. Math., vol.~270, Birkh\"auser Boston, Boston, MA, 2009,
503--531.

\bibitem[Per22]{Perry}
A.~Perry,
\emph{The integral Hodge conjecture for two-dimensional Calabi--Yau
categories},
Compos. Math. \textbf{158} (2022), no.~2, 287--333.

\bibitem[PV12]{PolishchukVaintrob}
A.~Polishchuk and A.~Vaintrob,
\emph{Chern characters and Hirzebruch--Riemann--Roch formula for matrix
factorizations},
Duke Math. J. \textbf{161} (2012), no.~10, 1863--1926.

\bibitem[PV16]{PolishchukVaintrobCohFT}
A.~Polishchuk and A.~Vaintrob,
\emph{Matrix factorizations and cohomological field theories},
J. Reine Angew. Math. \textbf{714} (2016), 1--122.

\bibitem[Pre11]{Preygel}
A.~Preygel,
\emph{Thom--Sebastiani and duality for matrix factorizations},
arXiv:1101.5834, 2011.

\bibitem[Ran80]{Ran}
Z.~Ran,
\emph{Cycles on Fermat hypersurfaces},
Compos. Math. \textbf{42} (1980), no.~1, 121--142.

\bibitem[Shi79]{Shioda}
T.~Shioda,
\emph{The Hodge conjecture for Fermat varieties},
Math. Ann. \textbf{245} (1979), no.~2, 175--184.

\bibitem[Squ84]{Squier}
C.~C.~Squier,
\emph{The Burau representation is unitary},
Proc. Amer. Math. Soc. \textbf{90} (1984), no.~2, 199--202.

\bibitem[Ste77]{Steenbrink}
J.~H.~M.~Steenbrink,
\emph{Mixed Hodge structure on the vanishing cohomology},
in \emph{Real and Complex Singularities}, Sijthoff and Noordhoff,
Alphen aan den Rijn, 1977, 525--563.

\bibitem[Ven14]{Venkataramana}
T.~N.~Venkataramana,
\emph{Image of the Burau representation at $d$-th roots of unity},
Ann. of Math. (2) \textbf{179} (2014), no.~3, 1041--1083.

\bibitem[Voi02]{Voisin}
C.~Voisin,
\emph{Hodge Theory and Complex Algebraic Geometry I},
Cambridge Studies in Advanced Mathematics, vol.~76,
Cambridge University Press, Cambridge, 2002.

\bibitem[Zuc77]{Zucker}
S.~Zucker,
\emph{The Hodge conjecture for cubic fourfolds},
Compos. Math. \textbf{34} (1977), no.~2, 199--209.

\enlargethispage{5\baselineskip}
\end{thebibliography}
\end{document}